\documentclass[12pt,a4paper]{article}
\usepackage{amsmath}
\usepackage{amssymb}
\usepackage{amsthm}
\usepackage{amsfonts}
\usepackage{latexsym}

\theoremstyle{plain}
\newtheorem{thm}{Theorem}[section]
\newtheorem{cor}{Corollary}[section]
\newtheorem{lem}{Lemma}[section]
\newtheorem{prop}{Proposition}[section]
\theoremstyle{definition}
\newtheorem{defn}{Definition}[section]

\theoremstyle{remark}

\newtheorem{rem}{Remark}[section]

\title{Scaling asymptotics for\\ quantized Hamiltonian flows}
\author{Roberto Paoletti\footnote{\noindent{\bf Address:}
Dipartimento di Matematica e Applicazioni, Universit\`a degli Studi
di Milano Bicocca, Via R. Cozzi 53, 20125 Milano,
Italy; {\bf e-mail}: roberto.paoletti@unimib.it }}
\date{}

\begin{document}

\maketitle
\begin{abstract}
In recent years, the near diagonal asymptotics of the equivariant components of the
Szeg\"{o} kernel of a positive line bundle on a compact symplectic manifold have been
studied extensively by many authors. As a natural generalization of this
theme, here we consider the local scaling asymptotics of the quantization of a Hamiltonian
symplectomorphism, and specifically how they concentrate on the graph of the underlying classical map.
\end{abstract}

\section{Introduction}

Suppose that $M$ is a connected d-dimensional
complex projective manifold, and let $(B,h)$ be a positive Hermitian line bundle
on $M$. Thus $B$ is ample as an holomorphic line bundle, and $h$ is an Hermitian metric
on $B$, such that the unique
compatible connection $\nabla$ on $(B,h)$ has
curvature $-2i\,\omega$, where $\omega$ is a K\"{a}hler form.
If $B^\vee$ is the dual line bundle, let $B^\vee\supseteq X\stackrel{\pi}{\rightarrow}M$ be the unit circle bundle;
the connection 1-form $\alpha$ is a contact form on $X$.

These choices determine natural volume forms $dV_M=:(1/\mathrm{d}!)\,\omega^{\wedge \mathrm{d}}$
on $M$ and $d\mu_X=:(1/2\pi)\,\alpha\,\wedge \pi^*(dV_M)$ on $X$, respectively, hence induce
Hermitian structures on the vector spaces $H^0\left(M,B^{\otimes k}\right)$
of global holomorphic sections of the tensor powers $B^{\otimes k}$, for $k=0,1,2,\ldots$.
The Hardy space
$H(X)\subseteq L^2(X)$ is unitarily isomorphic in a natural manner to the Hilbert space
direct sum of the $H^0\left(M,B^{\otimes k}\right)$'s, and
$H^0\left(M,B^{\otimes k}\right)$ corresponds to the $k$-th isotype $H(X)_k\subseteq H(X)$
for the circle action on $X$ \cite{sz}.

In geometric quantization, the symplectic manifold $(M,2\omega)$ is viewed as a \lq classical
phase space\rq, and the Hilbert spaces $H^0\left(M,B^{\otimes k}\right)$ as corresponding \lq quantum spaces\rq\,
at Planck's constant $\hbar =1/k$; the semiclassical regime corresponds to letting $k\rightarrow +\infty$.
A basic theme in this setting is the quantization of Hamiltonian functions and their Hamiltonian flows
(see, e.g., \cite{ber}, \cite{bg}, \cite{z-index}).

Consider a classical observable on $M$, given by a $\mathcal{C}^\infty$ function $f:M\rightarrow \mathbb{R}$,
with Hamiltonian vector field $\upsilon_f$, and corresponding
flow $\phi^M_\tau:M\rightarrow M$ ($\tau\in \mathbb{R}$); thus $\phi^M:\tau\mapsto \phi^M_\tau$ is
a 1-parameter group of Hamiltonian symplectomorphisms. One regards the self-adjoint
Toeplitz operator associated to
$f$, $T_f:H(X)\rightarrow H(X)$, as the quantization of $f$;
explicitly, $T_f=:\Pi\circ M_f\circ \Pi$, where $\Pi:L^2(X)\rightarrow H(X)$ is the orthogonal projection
(the so-called Szeg\"{o}
projector),
and $M_f$ is multiplication
by $f$ (pulled-back to $X$). Being $S^1$-invariant,
$T_f$ restricts to \lq quantum observables\rq\,
 $T_f^{(k)}:H(X)_k\rightarrow H(X)_k$. On the other hand, the quantization of $\phi^M_\tau$ should be a
family of $S^1$-invariant unitary operators $\Phi_\tau:H(X)\rightarrow H(X)$, asymptotically related to
the dynamics of $\phi^M_\tau$.

There exists a contact vector field $\widetilde{\upsilon}_f$ on $X$ lifting $\upsilon_f$ \cite{wein},
which depends on $f$ (and not just on $\upsilon_f$).
Consequently,
$\phi^M$ lifts to a 1-parameter group $\phi^X:\tau\mapsto \phi^X_\tau$ of contactomorphisms of $(X,\alpha)$;
pull-back determines a unitary action $\left(\phi^X_{-\tau}\right)^*:L^2(X)\rightarrow L^2(X)$.

When $\phi^M_\tau$
is holomorphic, $\left(\phi^X_{-\tau}\right)^*$ preserves $H(X)$, and
the restriction is a quantization
of $\phi^M$. Thus one sets $\Phi_\tau=:\Pi\circ \left(\phi^X_{-\tau}\right)^*\circ \Pi:H(X)\rightarrow H(X)$
in this case.

However, in general $\left(\phi^X_{-\tau}\right)^*\big(H(X)\big)\varsubsetneq
H(X)$, and $\Pi\circ \left(\phi^X_{-\tau}\right)^*\circ \Pi:H(X)\rightarrow H(X)$ is
not unitary. Nonetheless, in the setting of Fourier-Hermite distributions
 \cite{bg}, Zelditch proved that there exists a canonical family of invariant zeroth order Toeplitz operators
 $R_\tau:H(X)\rightarrow H(X)$, such that $\Phi_\tau=:R_\tau\circ \Pi\circ \left(\phi^X_{-\tau}\right)^*\circ \Pi$
 is indeed a unitary automorphisms of $H(X)$ (essentially), and computed the leading symbol of $R_\tau$
 by the symbolic calculus of symplectic spinors \cite{z-index}.

Here we approach similar issues by the local scaling asymptotics of the distributional
kernels of operators
of the same general form as $\Phi_\tau$. Much attention has been drawn in recent years by
the near-diagonal scaling asymptotics of the equivariant components of Szeg\"{o} kernels,
involving various authors and points of view; this paper is specifically related to
the approach in \cite{bsz}, \cite{sz}, based on the microlocal theory of \cite{bs}
(see for instance \cite{mm1} and \cite{mm2} for a different perspective). This line
of research originated from the so-called TYZ expansion, which first appeared in \cite{t}, \cite{c}, \cite{z}.

We shall first build on \cite{bsz} and \cite{sz}
to determine the equivariant scaling asymptotics of
\begin{equation}
\label{eqn:defn-U-tau}
U_\tau=:R_\tau\circ \Pi\circ \left(\phi^X_{-\tau}\right)^*\circ \Pi,
\end{equation}
now over the graph of $\phi^M_{-\tau}$ (Theorems \ref{thm:rapid-decay} and
\ref{thm:concentration-rate}); here $R_\tau$ is a general $\mathcal{C}^\infty$ family of
invariant zeroth order Toeplitz operators on $X$.
Then we shall determine the leading symbol of $R_\tau$ if $U_\tau$ is unitary (Corollary \ref{cor:condition-on-symbol});
in the reverse direction, we shall derive a version of the Zelditch unitarization Lemma
in \cite{z-index} (Corollary \ref{cor:esiste-unitario}).
To leading order, $U_\tau$ solves a Schr\"{o}dinger type equation
(Proposition \ref{prop:1-parameter-group}).

By definition, $R_\tau=\Pi\circ Q_\tau\circ \Pi$,
where $Q_\tau$ is a zeroth order invariant pseudodifferential operator of classical type on $X$; the symbol
$\varrho_\tau$ of $R_\tau$ is the restriction of the symbol of $Q_\tau$ to the closed symplectic cone sprayed by
the connection form,
$$
\Sigma=:\big\{\big(x,r\alpha_x):\,x\in X,\,r>0\big\}\subseteq T^*X\setminus\{0\}.
$$
Being homogeneous of degree zero and $S^1$-invariant, $\varrho_\tau$ is really a smooth function on $M$.

Define $U_\tau$ by (\ref{eqn:defn-U-tau}). Identifying densities, half-densities and functions by the previous choices,
also denote by $U_\tau\in \mathcal{D}'(X\times X)$ the Schwartz kernel of $U_\tau$. By invariance, $U_\tau$ restricts to
operators $U_{\tau,k}:H(X)_k\rightarrow H(X)_k$. If $k=0,1,2,\ldots$ and
$\left\{s_{kj}\right\}_{j=1}^{N_k}$ is an orthonormal basis of $H(X)_k$,
the corresponding distributional kernels are $U_{\tau,k}=\sum_{j=1}^{N_k}U_{\tau}\left(s_{kj}\right)\boxtimes \overline{s_{kj}}
\in \mathcal{C}^\infty(X\times X)$,
and $U_{\tau}(x,y)=\sum_{k\ge 0}U_{\tau,k}(x,y)$. More explicitly,
$$
U_{\tau,k}(x,y)=\sum_{j=1}^{N_k}U_{\tau}\left(s_{kj}\right)(x)\cdot \overline{s_{kj}(y)}
.$$

As in the case of the Szeg\"{o} kernel, the following scaling asymptotics for $U_{\tau,k}$ are expressed in terms
of Heisenberg local coordinates on $X$; these are precisely defined
in \cite{sz}. A system of Heisenber local coordinates centered at some $x\in X$
is built of a system of
preferred local coordinates on $M$, centered at $m=\pi(x)$ (meaning that the symplectic and complex structure are
the standard ones at the origin), and a preferred local section $e_L$ of $L$ at $m=\pi(x)$ (this is a prescription on the second order jet
of $e_L$ at $m$). In particular, a system of Heisenberg local coordinates centered at $x$
induces unitary isomorphism $T_mM\cong \mathbb{C}^\mathrm{d}$ and $T_xX\cong \mathbb{R}\times \mathbb{C}^\mathrm{d}$;
in the latter, $\mathbb{R}\times \{\mathbf{0}\}$ and $\{0\}\times \mathbb{C}^\mathrm{d}$ correspond
to the vertical and horizontal tangent spaces, respectively.

In Heisenberg local coordinates, the equivariant scaling asymptotics of
Szeg\"{o} kernels exhibit their universal nature. If $\gamma$ is a system of Heisenberg local
coordinates centered at $x$, following \cite{sz} we shall let
$x+(\theta,\mathbf{v})=:\gamma(\theta,\mathbf{v})$ if $\theta\in (-\pi,\pi)$, and $\mathbf{v}\in \mathbb{R}^{2\mathrm{d}}$
is sufficiently small; in the same range, we shall also write
$x+\mathbf{v}=:x+(0,\mathbf{v})$. If $\vartheta\in (-\pi,\pi)$ the action $r_\vartheta:X\rightarrow X$ of $e^{i\vartheta}\in S^1$
is expressed by a translation by $\vartheta$ in the angular coordinate, that is, $r_\vartheta\big(x+(\theta,\mathbf{v})\big)=x+(\vartheta+\theta,\mathbf{v})$ wherever defined; furthermore,
 $m+\mathbf{v}=:\pi\big(x+(\theta,\mathbf{v})\big)$ is the underlying system of preferred
local coordinates at $m$.
Given the built-in unitary isomorphism $T_mM\cong \mathbb{C}^\mathrm{d}$,
we shall also use the expressions $x+(\theta,\mathbf{v})$ and $x+\mathbf{v}$ for $\mathbf{v}\in T_mM$ of suitably small
norm.

Fix $x\in X$ and $\tau\in \mathbb{R}$, and set $x_\tau=:\phi^X_{-\tau}(x)$; if $m=:\pi(x)$
and $m_\tau=:\pi(x_\tau)$, then $m_\tau=\phi^M_{-\tau}(m)$.
Choose Heisenberg local coordinates centered
at $x$ and $x_\tau$ respectively; then the differential
$d_m\phi^M_{-\tau}:T_mM\rightarrow T_{m_\tau}M$ corresponds to a $2\mathrm{d}\times 2\mathrm{d}$
symplectic matrix $A_{\tau,m}$. A change in Heisenberg local
coordinates at $x$ and $x_\tau$ will turn $A_{\tau,m}$ into
$A_{\tau,m}'=:R\,A_{\tau,m}\,S^t$, where $R$ and $S$ are unitary (that is,
symplectic and orthogonal).

\begin{defn}
For $\tau\in \mathbb{R}$, the \textit{saturated graph} of $\phi^X_{-\tau}$ is
$$
\widetilde{\mathrm{graph}}\left(\phi^X_{-\tau}\right)
=:(\pi\times \pi)^{-1}\left(\mathrm{graph}\left(\phi^M_{-\tau}\right)\right)\subseteq X\times X.
$$
\end{defn}

Thus $ \widetilde{\mathrm{graph}}\left(\phi^X_{-\tau}\right)$ is the saturation of
$\mathrm{graph}\left(\phi^X_{-\tau}\right)$ under the $S^1$-action.
In other words, $(x,y)\in \widetilde{\mathrm{graph}}\left(\phi^X_{-\tau}\right)$ if and only if
$y=r_\vartheta(x_\tau)$ for some $e^{i\vartheta}\in S^1$.

As $k\rightarrow +\infty$, the kernel $U_{\tau,k}$ concentrates on $ \widetilde{\mathrm{graph}}\left(\phi^X_{-\tau}\right)$,
meaning that $U_{\tau,k}(x,y)=O\left(k^{-\infty}\right)$
uniformly in $(x,y)\in X\times X\setminus \widetilde{\mathrm{graph}}\left(\phi^X_{-\tau}\right)$.
More precisely, if $\mathrm{dist}_X$ is the Riemannian distance on $X$ then we have:

\begin{thm}
\label{thm:rapid-decay}
For any $D,\varepsilon>0$, uniformly in $(x,y)\in X\times X$ satisfying
$\mathrm{dist}_X\left(y,S^1\cdot x_\tau\right)\ge D\,k^{\varepsilon-\frac 12}$, we have
$U_{\tau,k}(x,y)=O\left(k^{-\infty}\right)$ as $k\rightarrow +\infty$.
\end{thm}

Let us analyze the rate at which $U_{\tau,k}$ concentrates on the saturated graph.
For any $e^{i\vartheta_1},\,e^{i\vartheta_2}\in S^1$ and $(x,y)\in X\times X$, we have
$U_{\tau,k}\big(r_{\vartheta_1}(x),r_{\vartheta_2}(y)\big)=
e^{ik(\vartheta_1-\vartheta_2)}\,U_{\tau,k}(x,y)$, so that without loss we may work in the neighborhood
of a given $(x,x_\tau)\in \mathrm{graph}\left(\phi^X_{-\tau}\right)$. Thus we may consider the
behavior of $U_{\tau,k}$ at points of the form $\big(x+(\vartheta_1,\mathbf{u}),x_\tau+(\vartheta_2,\mathbf{w})\big)$,
computed in systems of Heisenberg local coordinates centered at $x$ and $x_\tau$, respectively. Now
\begin{eqnarray}\label{eqn:equivarianza-U-k}
\lefteqn{U_{\tau,k}\big(x+(\vartheta_1,\mathbf{u}),x_\tau+(\vartheta_2,\mathbf{w})\big)}\\
&=&
U_{\tau,k}\big(r_{\vartheta_1}(x+\mathbf{u}),r_{\vartheta_2}(x_\tau+\mathbf{w})\big)=
e^{ik(\vartheta_1-\vartheta_2)}\,U_{\tau,k}\big(x+\mathbf{u},x_\tau+\mathbf{w}\big),\nonumber
\end{eqnarray}
so we need only consider pairs $\big(x+\mathbf{u},x_\tau+\mathbf{w}\big)$ converging to $(x,x_\tau)$.

In order to formulate our result, we need to define a certain quadratic function
$\mathcal{S}_A:\mathbb{R}^{2\mathrm{d}}\times \mathbb{R}^{2\mathrm{d}}\rightarrow \mathbb{C}$ associated to a symplectic matrix $A$.
Let
$$
J_0=:\left(
       \begin{array}{cc}
         0 & -I_d \\
         I_d & 0 \\
       \end{array}
     \right);
$$
thus $J_0$ represents the standard complex structure on $\mathbb{R}^{2\mathrm{d}}$, and
$-J_0$ the standard symplectic structure $\omega_0$.

\begin{defn}
\label{defn:key-bilinear-form}
Let $A$ be a symplectic $2\mathrm{d}\times 2\mathrm{d}$ matrix, and let
$A=O\,P$ be its polar decomposition; thus $O$ is orthogonal and symplectic,
hence unitary, and $P$ is symmetric positive definite and symplectic.
Then the following matrices are symmetric:
$$
Q_A=I+P^2,\,\,\,\mathcal{P}_A=:O\,Q_A^{-1}\,O^t,\,\,\,
\mathcal{R}_A=:O\,\left(I-P^2\right)\,Q_A^{-1}\,J_0\,O^t.
$$
For $\mathbf{u},\mathbf{w}\in  \mathbb{R}^{2\mathrm{d}}$, let
$L(\mathbf{u},\mathbf{w})=:A\mathbf{u}-\mathbf{w}$. Define $\mathcal{S}_A:\mathbb{R}^{2\mathrm{d}}\times \mathbb{R}^{2\mathrm{d}}
\rightarrow \mathbb{C}$ by setting
$$
\mathcal{S}_A(\mathbf{u},\mathbf{w})=:-L(\mathbf{u},\mathbf{w})^t\,\left[\mathcal{P}_A+\frac i2\,\mathcal{R}_A\right]\,
L(\mathbf{u},\mathbf{w})-i\,\omega_0(A\mathbf{u},\mathbf{w}).
$$
\end{defn}

For example, when $A=O$ is unitary (that is, $P=I$) we have $\mathcal{P}_A=\frac 12\,I_d$, and
$$
\mathcal{S}_A(\mathbf{u},\mathbf{w})=:-\frac 12\,\|A\mathbf{u}-\mathbf{w}\|^2-i\,\omega_0(A\mathbf{u},\mathbf{w})=
\psi_2(A\mathbf{u},\mathbf{w}),
$$
where $\psi_2$ is the universal exponent in the equivariant Szeg\"{o} kernel asymptotics \cite{sz}.

If $R$ and $S$ are unitary matrices, we have
$\mathcal{S}_{RAS^t}(S\mathbf{u},R\mathbf{w})= S_A(\mathbf{u},\mathbf{w})$. Thus if $m_\tau=:\phi^M_{-\tau}(m)$,
and $A=A_{\tau,m}$
represents $d_m\phi^M_{-\tau}:T_mM\rightarrow T_{m_\tau}M$,
then $\mathcal{S}_{A}$ does not depend on the choice of Heisenberg
local coordinates, and is
well-defined as a function
$$
\mathcal{S}_{\tau,m}:T_mM\times T_{m_\tau}M
\longrightarrow \mathbb{C}.
$$

Similarly, $\nu:\mathbb{R}\times M\rightarrow \mathbb{R}$ given by
\begin{equation}
\label{eqn:defn-nu}
\nu(\tau,m)=:\sqrt{\det \left(Q_{A_{\tau,m}}\right)}
\end{equation}
is well-defined. If $d_m\phi^M_\tau$ is unitary, $\nu(\tau,m)=2^\mathrm{d}$. Notice that, with
$A=A_{\tau,m}$,
\begin{equation}
\label{eqn:nu-compare}
\nu(\tau,m)=\det \left(I+A^t\,A\right)^{1/2}=\det\big(A\,J_0+J_0\,A\big)^{1/2}.
\end{equation}

As a further piece of notation, let $T_{\tau,m}\subseteq T_mM\times T_{m_\tau}M$ be the tangent space at
$\left(m,m_\tau\right)$ to $\mathrm{graph}\left(\phi^M_{-\tau}\right)$. In Heisenberg local coordinates,
this is
$$
T_{\tau,m}=\mathrm{graph}(A)=:\left\{(\mathbf{u},\mathbf{w})\in \mathbb{C}^\mathrm{d}\times \mathbb{C}^\mathrm{d}:
\,A\mathbf{u}=\mathbf{w}\right\}.
$$
Finally, let $N_{\tau,m}=:T_{\tau,m}^\perp\subseteq T_mM\times T_{m_\tau}M$ be its orthocomplement for the
Riemannian metric on $M\times M$.

\begin{thm}
\label{thm:concentration-rate}
Let $R_\tau$ be invariant zeroth order Toeplitz operators on $X$, with symbol
$\varrho_\tau\in \mathcal{C}^\infty(M)$, and define
$U_\tau$ by (\ref{eqn:defn-U-tau}).
Suppose $x\in X$, $x_\tau=:\phi^X_{-\tau}(x)$, $m=:\pi(x)$. Fix Heisenberg local coordinates centered at $x$ and $x_\tau$, respectively. Let $E>0$ be a constant.
Then, uniformly in $\mathbf{u}\in T_mM$ and $\mathbf{w}\in T_{m_\tau}M$ with $\|\mathbf{u}\|,\,\|\mathbf{w}\|\le
E\, k^{1/9}$ and $(\mathbf{u},\mathbf{w})\in N_{\tau,m}$, as $k\rightarrow +\infty$ we have
\begin{eqnarray*}
\lefteqn{U_{\tau,k}\left(x+\frac{\mathbf{u}}{\sqrt{k}},x_\tau+\frac{\mathbf{w}}{\sqrt{k}}\right)}\\
&\sim&\varrho_\tau(m)\,\left(\frac k\pi\right)^\mathrm{d}\,\frac{2^{\mathrm{d}}}{\nu(\tau,m)}\,
\cdot e^{\mathcal{S}_{\tau,m}(\mathbf{u},\mathbf{w})}\cdot \left(1+\sum_{j\ge 1}k^{-j/2}\,a_j(m,\tau,\mathbf{u},\mathbf{w})\right),
\end{eqnarray*}
where $a_j(m,\tau,\mathbf{u},\mathbf{w})$
is a polynomial in $\mathbf{u}$ and $\mathbf{w}$, depending smoothly on $m,\tau$.

In addition, the polynomial $a_j(m,\tau,X,Y)$ has the same parity as $j$.
\end{thm}

Explicitly, the last claim is that
$a_j(m,\tau,-X,-Y)=(-1)^j\,a_j(m,\tau,X,Y)$.

In particular, Theorem \ref{thm:concentration-rate} describes an exponential decay of the
rescaled kernel $U_{\tau,k}\left(x+\mathbf{u}/\sqrt{k},x_\tau+\mathbf{w}/\sqrt{k}\right)$
along normal directions to the graph; the same will hold under the general transversality assumption
$A\mathbf{u}\neq \mathbf{w}$.

\begin{cor}\label{cor:condition-on-symbol}
If $U_{\tau,k}$ is unitary for $k\gg 0$, then
\begin{equation}
\label{eqn:condition-on-symbol}
\big|\varrho_{\tau}(m)\big|=2^{-\mathrm{d}/2}\,\sqrt{\nu (\tau,m)}.
\end{equation}
\end{cor}

The hypothesis in Corollary \ref{cor:condition-on-symbol} means that $U_\tau$
is unitary, as an endomorphism of $H(X)$, on the complement of a finite dimensional subspace;
it is obviously satisfied if $U_\tau$ itself is unitary.

We can give an analogue of the unitarization Lemma of \cite{z-index}:

\begin{cor}
\label{cor:esiste-unitario}
There exists a $\mathcal{C}^\infty$ family $R_\tau$ of zeroth order Toeplitz operators
$R_\tau$ such that if $U_\tau$ is defined by (\ref{eqn:defn-U-tau}), then
$$
U_{\tau,k}\circ U_{\tau,k}^*=\Pi_k+O\left(k^{-\infty}\right),\,\,\,\,\,\,\,\,
U_{\tau,k}^*\circ U_{\tau,k}=\Pi_k+O\left(k^{-\infty}\right).
$$
\end{cor}

It follows from the proof of Corollary \ref{cor:esiste-unitario}
that there is a canonical choice for $R_\tau$, up to smoothing operators.

\begin{rem}
The functional argument on page 327 of \cite{z-index} shows that $U_\tau$ may be modified so as to assume that it is
actually unitary on the complement of a finite dimensional subspace of $H(X)$.
\end{rem}

\begin{rem}
If $A=A_{\tau,m}$, by (\ref{eqn:nu-compare})
we can rewrite the right hand side of (\ref{eqn:condition-on-symbol}) as
$$
2^{-\mathrm{d}/2}\,\det\big(A\,J_0+J_0\,A\big)^{1/4},
$$
which tallies with the multiplier determined in \S 6 of \cite{d} for the linear case.
\end{rem}

Finally, to leading order $U_\tau$ satisfies the Shr\"{o}dinger equation associated to $f$.
Namely, let $D_\theta=:-i\,\partial/\partial\theta$, where $\partial/\partial\theta$
is the generator of the $S^1$-action on $X$, and $\widetilde{T}_f=:D_\theta\circ T_f$.
Thus $D_\theta$ is the \lq number operator\rq\, equal to $k\,\mathrm{id}_{H(X)_k}$ on
$H(X)_k$, and $\widetilde{T}_f$ is a self-adjoint invariant first order Toeplitz operator,
and its restriction to $H(X)_k$ is $\widetilde{T}_f^{(k)}=k\,T_f^{(k)}$.

\begin{prop}
\label{prop:1-parameter-group}
In the situation of Theorem \ref{thm:concentration-rate},
\begin{eqnarray*}
\lefteqn{\left.\frac{d }{d\tau}U_{\tau,k}\right|_{\tau_0}
\left(x+\frac{\mathbf{u}}{\sqrt{k}},x_{\tau_0}+\frac{\mathbf{w}}{\sqrt{k}}\right)}\\
&=&\left(i\,\widetilde{T}_f^{(k)}\circ U_{\tau_0,k}\right)\left(x+\frac{\mathbf{u}}{\sqrt{k}},x_{\tau_0}+\frac{\mathbf{w}}{\sqrt{k}}\right)
+ e^{\mathcal{S}_{\tau_0,m}(\mathbf{u},\mathbf{w})}\cdot O\left(k^{\mathrm{d}+1/2}\right).
\end{eqnarray*}
\end{prop}

Under favorable conditions
Theorems \ref{thm:rapid-decay} and \ref{thm:concentration-rate} also yield an asymptotic expansion for
the trace of $U_{\tau,k}$. Let us consider the simplest case where $U_{\tau,k}$ only has isolated and
non-degenerate fixed points $m_1,\ldots,m_r$. Then on the one hand by Theorem \ref{thm:rapid-decay}
$$
\mathrm{trace}\big(U_{\tau,k}\big)=\int_XU_{\tau,k}(x,x)\,dV_X(x)\sim
\sum_{i=1}^r\int_{X_i}\gamma_i(x)\,U_{\tau,k}(x,x)\,dV_X(x),
$$
where $\gamma_i$ is an invariant bump function, supported near $\pi^{-1}(x_i)$
and identically equal to one in a small neighborhood of $\pi^{-1}(x_i)$.

On the other hand, working in rescaled Heisenberg local coordinates centered at some
$x_i$ lying over $m_i$, and writing $\mathcal{S}_{\tau,m_i}(\mathbf{u},\mathbf{u})=-(1/2)\,\mathbf{u}^tS_{i}\mathbf{u}$
for a symmetric matrix $S_{i}$ with positive real part, by Theorem \ref{thm:concentration-rate} we have
\begin{eqnarray*}
\lefteqn{\int_{X_i}\gamma_i(x)\,U_{\tau,k}(x,x)\,dV_X(x)}\\
& \sim &k^{-\mathrm{d}}\cdot\int_{\mathbb{C}^\mathrm{d}}
\gamma_i\left(x_i+\frac{\mathbf{u}}{\sqrt{k}}\right)\,
U_{\tau,k}\left(x_i+\frac{\mathbf{u}}{\sqrt{k}},
x_i+\frac{\mathbf{u}}{\sqrt{k}}\right)\,d\mathbf{v}\\
&=&k^{-\mathrm{d}}\varrho_\tau(m_i)\,\left(\frac k\pi\right)^\mathrm{d}\,\frac{2^\mathrm{d}}{\nu(\tau,m_i)}\,
\int_{\mathbb{C}^\mathrm{d}}e^{-\frac 12\,\mathbf{u}^tS_{i}\mathbf{u}}\,d\mathbf{u} +\mathrm{L.O.T.}\\
&=&\frac{2^{2\mathrm{d}}}{\nu(\tau,m_i)}\,\det\left(S_i\right)^{-1/2}+\mathrm{L.O.T.}
\end{eqnarray*}
(here L.O.T. = lower order terms). Similar expansions may be obtained for higher dimensional symplectic fixed
loci, adapting the arguments in \cite{pao-trace}, but won't be discussed here.

For the sake of simplicity, we have restricted our exposition to the
complex projective setting; however, by the microlocal theory developed in \cite{sz},
the previous results can be generalized to the case of almost
complex symplectic manifolds.

\section{Proof of Theorem \ref{thm:rapid-decay}.}

Let $\Pi_\tau=:\left(\phi^X_{-\tau}\right)^*\circ \Pi$. In terms of Schwartz kernels,
$$
\Pi_\tau=\left( \phi^X_{-\tau}\times \mathrm{id}_X\right)^*(\Pi).
$$
Then $U_\tau=R_\tau\circ \Pi_\tau$, and since $R_\tau$ and $\Pi_\tau$ are $S^1$-invariant, they
preserve each $S^1$-equivariant summand $L^2(X)_k\subseteq L^2(X)$. Therefore, the restriction
$U_{\tau,k}:H(X)_k\rightarrow H(X)_k$ is a composition $U_{\tau,k}=R_{\tau,k}\circ \Pi_{\tau,k}$, where
$R_{\tau,k}$ and  $\Pi_{\tau,k}$ are the restrictions of
$U_\tau$ and $R_\tau$. In fact, since $R_\tau$ and $\Pi_\tau$ commute with the orthogonal
projection onto $L^2(X)_k$, we have
\begin{equation}
\label{eqn:composizione-equivariante}
U_{\tau,k}(x,y)=R_{\tau,k}\circ \Pi_{\tau,k}=R_{\tau,k}\circ \Pi_{\tau}=R_{\tau}\circ \Pi_{\tau,k}.
\end{equation}
Using distributional kernels, we can rewrite (\ref{eqn:composizione-equivariante}) in the form
\begin{eqnarray}
\label{eqn:composizione-integrata}
U_{\tau,k}(x,y)&=&\int_X R_{\tau,k}(x,z) \Pi_{\tau,k}(z,y)\,d\mu_X(z)\nonumber\\
&=&\int_X R_{\tau,k}(x,z) \Pi_{k}\left(\phi^X_{-\tau}(z),y\right)\,d\mu_X(z).
\end{eqnarray}
where clearly $R_{\tau,k},\,\Pi_{\tau,k}\in \mathcal{C}^\infty(X\times X)$. Furthermore, since
$\Pi$ is $\mathcal{C}^\infty$ on $X\times X\setminus \mathrm{diag}(X)$
by \cite{f} \cite{bs}, so is $R_\tau$; therefore,
$R_{\tau,k}(x,z)=O\left(k^{-\infty}\right)$ as $k\rightarrow +\infty$, uniformly on the locus where
$\mathrm{dist}_X\left(x,S^1\cdot z\right)\ge \delta$, for any fixed $\delta>0$. Similarly,
$\Pi_{k}\left(\phi^X_{-\tau}(z),y\right)=O\left(k^{-\infty}\right)$
uniformly for $\mathrm{dist}_X\left(y,S^1\cdot\phi^X_{-\tau}(z)\right)\ge \delta$. It is well-known that
$R_{\tau,k}(x,x)=\Pi_{k}(x,x)=O\left(k^\mathrm{d}\right)$.

\begin{lem}
For any $\epsilon>0$, we have $U_{\tau,k}(x,y)=O\left(k^{-\infty}\right)$ as $k\rightarrow +\infty$,
uniformly for $\mathrm{dist}_X\left(y,\phi^X_{-\tau}(x)\right)\ge \epsilon$.
\end{lem}

\begin{proof}
Let $C,c>0$ be such that for any $x_1,x_2\in X$ we have
\begin{equation}
\label{eqn:boundf-on-distance}
c\cdot\mathrm{dist}_X\left(x_1,x_2\right)\le \mathrm{dist}_X\left(\phi^X_{-\tau}(x_1),\phi^X_{-\tau}(x_2)\right)
\le C\cdot\mathrm{dist}_X\left(x_1,x_2\right),
\end{equation}
and similarly for any $m_1,m_2\in M$
\begin{equation}
\label{eqn:boundf-on-distance-M}
c\cdot\mathrm{dist}_M\left(m_1,m_2\right)\le \mathrm{dist}_M\left(\phi^M_{-\tau}(m_1),\phi^M_{-\tau}(m_2)\right)
\le C\cdot\mathrm{dist}_M\left(m_1,m_2\right).
\end{equation}

Choose $\epsilon>0$ arbitrarily small, and suppose
$\mathrm{dist}_X\left(y,S^1\cdot \phi^X_{-\tau}(x)\right)\ge \epsilon$. Define
$$
V=:\left\{z\in X: \mathrm{dist}_X\left(z, S^1\cdot x\right)>\frac{\epsilon}{3C}\right\}=
\left\{z\in X: \mathrm{dist}_M\big(\pi(z), \pi(x)\big)>\frac{\epsilon}{3C}\right\},
$$
$$
W=:\left\{z\in X: \mathrm{dist}_X\left(z, S^1\cdot x\right)<\frac{\epsilon}{2C}\right\}=
\left\{z\in X: \mathrm{dist}_M\big(\pi(z), \pi(x)\big)<\frac{\epsilon}{2C}\right\}.
$$
Then $\{V,W\}$ is an invariant open cover of $X$;
let $\{1-\varrho,\varrho\}$ be an invariant partition of unity subordinate to it.
We can rewrite (\ref{eqn:composizione-integrata}) as
\begin{eqnarray}
\label{eqn:composizione-integrata-partizionata}
U_{\tau,k}(x,y)&=&\int_V \big(1-\varrho(z)\big)\cdot R_{\tau,k}(x,z) \Pi_{k}\left(\phi^X_{-\tau}(z),y\right)\,d\mu_X(z)\nonumber\\
&&+\int_W \varrho(z)\cdot R_{\tau,k}(x,z) \Pi_{k}\left(\phi^X_{-\tau}(z),y\right)\,d\mu_X(z).
\end{eqnarray}

Uniformly in $z\in V$, we have on the one hand $\Pi_{k}\left(\phi^X_{-\tau}(z),y\right)=O\left(k^{\mathrm{d}}\right)$, and on the
other $R_{\tau,k}(x,z)=O\left(k^{-\infty}\right)$. Therefore, the first summand on the right hand side of
(\ref{eqn:composizione-integrata-partizionata}) rapidly decreasing as
$k\rightarrow +\infty$.

Uniformly in $z\in W$, we have on the one hand $R_{\tau,k}(x,z)=O\left(k^{d}\right)$, and on the other
\begin{eqnarray}
\label{eqn:boundf-on-distance-1}
\mathrm{dist}_X\left(y,\phi^X_{-\tau}(z)\right)&\ge& \mathrm{dist}_X\left(y,\phi^X_{-\tau}(x)\right)-
\mathrm{dist}_X\left(\phi^X_{-\tau}(x),\phi^X_{-\tau}(z)\right)\\
&\ge&\mathrm{dist}_X\left(y,\phi^X_{-\tau}(x)\right)-
C\cdot \mathrm{dist}_X\left(x,z\right)> \frac 12\,\epsilon;\nonumber
\end{eqnarray}
therefore, $\Pi_{k}\left(\phi^X_{-\tau}(z),y\right)=O\left(k^{-\infty}\right)$ on $W$, and the second summand is
also rapidly decreasing.
\end{proof}

We may thus assume that $(x,y)$ lies in an arbitrary small invariant tubular neighborhood of
$\widetilde{\mathrm{graph}}\left(\phi^X_{-\tau}\right)$, that is,
$\mathrm{dist}_X\left(y,S^1\cdot x_\tau\right)<\epsilon$ for some small $\epsilon>0$,
where $x_\tau=:\phi^X_{-\tau}(x)$.
In view of (\ref{eqn:equivarianza-U-k}), we may as well assume $\mathrm{dist}_X\left(y,x_\tau\right)<\epsilon$,
hence that $y=x_\tau+O(\epsilon)$ in any given system of local coordinates.

Let us define
\begin{eqnarray*}
V'&=:&\left\{z\in X:\,\mathrm{dist}_X\left(z,S^1\cdot x\right)>\frac 2c\cdot \epsilon\right\},\\
W'&=:&\left\{z\in X:\,\mathrm{dist}_X\left(z,S^1\cdot x\right)<\frac 3c\cdot\epsilon\right\},
\end{eqnarray*}
where $c$ is as in (\ref{eqn:boundf-on-distance}).
Let $\{1-\varrho',\varrho'\}$ be an invariant partition of unity on $X$, subordinate to the
open cover $\{V',W'\}$.
In distributional short-hand, using the last equality in (\ref{eqn:composizione-integrata}), we get
\begin{eqnarray}
\label{eqn:composizione-integrata-partizionata-1}
U_{\tau,k}(x,y)&=&\int_{V'} \big(1-\varrho'(z)\big)\cdot R_{\tau}(x,z) \Pi_{k}\left(\phi^X_{-\tau}(z),y\right)\,d\mu_X(z)\nonumber\\
&&+\int_{W'} \varrho'(z)\cdot R_{\tau}(x,z) \Pi_{k}\left(\phi^X_{-\tau}(z),y\right)\,d\mu_X(z).
\end{eqnarray}

\begin{lem}
The first integral on the
right hand side of (\ref{eqn:composizione-integrata-partizionata-1}) is $O\left(k^{-\infty}\right)$.
\end{lem}

\begin{proof}
On the $S^1$-invariant open set $V'\subseteq X$, $R_\tau(x,\cdot)$ is $\mathcal{C}^\infty$ and uniformly bounded.
Furthermore, for $z\in V'$ we have
\begin{eqnarray*}
\mathrm{dist}_X\left(\phi^X_{-\tau}(z),y\right)&\ge&
\mathrm{dist}_X\left(\phi^X_{-\tau}(x),\phi^X_{-\tau}(z)\right)-\mathrm{dist}_X\left(\phi^X_{-\tau}(x),y\right)\\
&\ge&c\,\mathrm{dist}_X\left(x,z\right)-\mathrm{dist}_X\left(\phi^X_{-\tau}(x),y\right)> c\,\frac 2c\cdot \epsilon-\epsilon
=\epsilon.
\end{eqnarray*}
Therefore, $\Pi_{k}\left(\phi^X_{-\tau}(z),y\right)=O\left(k^{-\infty}\right)$ uniformly for $z\in V'$.
\end{proof}

It follows that as $k\rightarrow +\infty$
\begin{eqnarray}
\label{eqn:composizione-integrata-partizionata-2}
U_{\tau,k}(x,y)
&\sim&\int_{W'} \varrho'(z)\cdot R_{\tau}(x,z) \Pi_{k}\left(\phi^X_{-\tau}(z),y\right)\,d\mu_X(z),
\end{eqnarray}
where $\sim$ stands for 'equal asympotics as'.
Now in (\ref{eqn:composizione-integrata-partizionata-2}) $z$ is in a small $S^1$-invariant neighborhood of $x$, while $\phi^X_{-\tau}(z)$
and $y$ are in a small $S^1$-invariant neighborhood of $x_\tau=\phi^X_{-\tau}(x)$. Perhaps disregarding
a smoothing term not contributing to the asymptotics, we may now introduce in (\ref{eqn:composizione-integrata-partizionata-2})
the microlocal descriptions of $R_{\tau}$ and $\Pi$ as Fourier integral operators from \cite{bs}, and work in
Heisenberg local coordinates centered at $x$ and $x_\tau$, respectively.

More precisely, by the discussion in \cite{bs}, \cite{bsz}, \cite{sz} we have
\begin{equation}
\label{eqn:pi-as-fio}
\Pi\left(x',x''\right)=:\int_0^{+\infty}e^{iu\,\psi\left(x',x''\right)}\,s\left(u,x',x''\right)\,du
\end{equation}
and
\begin{equation}
\label{eqn:R-as-fio}
R_\tau\left(y',y''\right)=:\int_0^{+\infty}e^{it\,\psi\left(y',y''\right)}\,a_\tau\left(t,y',y''\right)\,dt;
\end{equation}
here $\psi$ is a complex phase of positive type, essentially determined by the
Taylor expansion of the metric along the diagonal, and $s$, $a_\tau$ are semiclassical symbols. More precisely,
\begin{equation}
\label{eqn:asymp-exp-amplitues}
s\left(u,x',x''\right)\sim\sum_{j\ge 0}u^{\mathrm{d}-j}\,s_j\left(x',x''\right),\,\,\,\,\,
a\left(t,x',x''\right)\sim\sum_{j\ge 0}t^{\mathrm{d}-j}\,a_j\left(x',x''\right),
\end{equation}
and since we are working in Heisenberg local coordinates centered at $x$ and $x_\tau$, respectively, we have
\begin{equation}
\label{eqn:amplitudes-leading-term}
a_0\left(x,x\right)=\varrho_\tau(m)\,(k/\pi)^{\mathrm{d}},\,\,\,\,s_0\left(x_\tau,x_\tau\right)=(k/\pi)^{\mathrm{d}}.
\end{equation}
Inserting (\ref{eqn:pi-as-fio})
and (\ref{eqn:R-as-fio}) in (\ref{eqn:composizione-integrata-partizionata-2}), and performing the rescaling
$t\mapsto k\,t$ and $u\mapsto ku$, we get
\begin{eqnarray}
\label{eqn:composizione-integrata-partizionata-fio}
U_{\tau,k}(x,y)\sim
\frac{k^2}{2\pi}\,\int_0^{+\infty}\int_0^{\infty}\int_{-\pi}^\pi\int_{W'}e^{i\,k\Psi_1}\,\mathcal{A}_1\,dt\,du\,d\vartheta\,d\mu_X(z),
\end{eqnarray}
where
\begin{equation}
\label{eqn:def-of-phase}
\Psi_1=:t\,\psi(x,z)
+u\,\psi\left(r_\vartheta(z_\tau),y\right)-\vartheta,
\end{equation}
where $z_\tau=:\phi^X_{-\tau}(z)$,
and
$$
\mathcal{A}_1=:a_\tau (kt,x,z)\,s\big(ku,r_\vartheta(z_\tau),y\big).
$$

On the diagonal, we have $d_{(x,x)}\psi=(\alpha_x,-\alpha_x)$; more generally,
$d_{(r_\theta(x),x)}\psi=\left(e^{i\theta}\alpha_{r_\theta(x)},-e^{i\theta}\alpha_{x}\right)$.
Working in Heisenberg local coordinates near $x$, we can write $z=x+(\theta,\mathbf{v})$, where
$\|\mathbf{v}\|<(6/c)\,\epsilon$, say; consequently, in Heisenberg local coordinates near $x_\tau$
we have
$z_\tau=x_\tau+\big(\theta,\mathbf{v}'_\tau\big)$, where again
$\left\|\mathbf{v}'_\tau\right\|=O(\epsilon)$.
In other words, $z=r_\theta(x)+O(\epsilon)$, $z_\tau=r_\theta(x_\tau)+O(\epsilon)$, and on the other hand
$y=x_\tau+O(\epsilon)$. Therefore,
\begin{equation}
\label{eqn:derivata-theta-psi}
\partial_\theta\Psi_1=t\,\left[-e^{-i\theta}+O(\epsilon)\right]+u\,\left[e^{i(\theta+\vartheta)}+O(\epsilon)\right],
\end{equation}
\begin{equation}
\label{eqn:derivata-vartheta-psi}
\partial_\vartheta\Psi_1=u\,\left[e^{i(\theta+\vartheta)}+O(\epsilon)\right]-1.
\end{equation}

It follows that
\begin{equation}
\label{eqn:bound-on-gradient}
\left\|\nabla_{\theta,\vartheta}\Psi_1\right\|\ge \sqrt{(u-t)^2+(u-1)^2}+O\left(\|(t,u)\|\cdot \epsilon\right).
\end{equation}

Therefore, if $\epsilon$ is sufficiently small then
$\left\|\nabla_{\theta,\vartheta}\Psi_1\right\|$ remains bounded away from zero
when $(u,t)$ does not belong to a small neighborhood of $(1,1)$, and it is $\ge (1/2)\|(u,t)\|$, say,
as $(u,t)\rightarrow \infty$.

We can rewrite (\ref{eqn:composizione-integrata-partizionata-fio}) in local coordinates
with $d\mu_X(z)=\mathcal{V}(\theta,\mathbf{v})\,d\theta\,d\mathbf{v}$. In addition, $\theta$ and $\vartheta$ are really local
coordinates on $S^1$ and therefore, upon introducing appropriate partitions on unity on the circle,
the corresponding integration may be implicitly interpreted as compactly supported.

Integrating by parts in $(\theta,\vartheta)$ we deduce from (\ref{eqn:bound-on-gradient})
that we only miss a negligible contribution to the asymptotics as $k\rightarrow +\infty$, if
integration in $(t,u)$ is restricted to a compact
neighborhood of $(1,1)$. Therefore,

\begin{lem}
Suppose $E\gg 0$, and let $\varrho_1\in \mathcal{C}^\infty _0\big((1/E,E)\big)$ be $\ge 0$ everywhere and
$\equiv 1$ on $(2/E,E/2)$. Then
\begin{eqnarray}
\label{eqn:composizione-integrata-compatta}
U_{\tau,k}(x,y)\sim
\frac{k^2}{2\pi}\,\int_{1/E}^{E}\int_{1/E}^{E}
\int_{-\pi}^\pi\int_{W'}e^{i\,k\Psi_1}\,\mathcal{A}_2\,dt\,du\,d\vartheta\,d\mu_X(z),
\end{eqnarray}
where $\mathcal{A}_2=:\mathcal{A}_1\cdot \varrho_1(t)\,\varrho_1(u)$.
\end{lem}

Given $D>0$ and with $C>0$ as in (\ref{eqn:boundf-on-distance}), let us consider the invariant open sets
\begin{equation*}
V'_k=:\left\{z\in W':\mathrm{dist}_X\left(z,S^1\cdot x\right)>\frac{D}{2C}\cdot k^{\varepsilon -1/2}\right\},
\end{equation*}
\begin{equation*}
W'_k=:\left\{z\in W':\mathrm{dist}_X\left(z,S^1\cdot x\right)<\frac{2D}{3C}\cdot k^{\varepsilon -1/2}\right\},
\end{equation*}
and let $\left\{1-\gamma_k,\,\gamma_k\right\}$ be an invariant partition of unity on $X$ subordinate to it;
we may assume that in local Heisenberg coordinates we have
$\gamma _k(z)=\gamma_1\left(k^{1/2-\epsilon}\,\|\mathbf{v}\|\right)$.

We can then rewrite (\ref{eqn:composizione-integrata-partizionata-fio}) as follows
\begin{eqnarray}
\label{eqn:composizione-integrata-partizionata-k}
U_{\tau,k}(x,y)&\sim&
\frac{k^2}{2\pi}\,
\int_{1/E}^{E}\int_{1/E}^{E}\int_{-\pi}^\pi\int_{V'_k}e^{i\,k\Psi_1}\,
\big(1-\gamma_k(z)\big)\,\mathcal{A}_1\,dt\,du\,d\vartheta\,d\mu_X(z)
\nonumber\\
&&+\frac{k^2}{2\pi}\,
\int_{1/E}^{E}\int_{1/E}^{E}\int_{-\pi}^\pi\int_{W'_k}e^{i\,k\Psi_1}\,\gamma_k(z)\,
\mathcal{A}_1\,dt\,du\,d\vartheta\,d\mu_X(z).
\end{eqnarray}

\begin{lem}
\label{lem:reduction-k}
The first summand on the right hand side of (\ref{eqn:composizione-integrata-partizionata-k}) is $O\left(k^{-\infty}\right)$.
\end{lem}

\begin{proof}
For $z\in V'_k$, given (\ref{eqn:def-of-phase}) and by Corollary 1.3 of \cite{bs} we have
$$
\big |\partial _t\Psi_1\big|=\big|\psi (x,z)\big|\ge \Im \psi (x,z)\ge C_1\,k^{2\varepsilon-1},
$$
where $C_1>0$ is an appropriate constant. The statement follows by iteratively integrating by parts in $dt$.
\end{proof}

Thus we are reduced to considering the second summand.

\begin{lem}
\label{lem:reduction-k-1}
Given that
$\mathrm{dist}_X\left(y,S^1\cdot x_\tau\right)>D\,k^{\varepsilon-1/2}$, the second summand on the right hand side of (\ref{eqn:composizione-integrata-partizionata-k}) is also $O\left(k^{-\infty}\right)$.
\end{lem}

\begin{proof}
Setting $n=:\pi(y)$ and $m=:\pi(x)$,
this may be rewritten
$\mathrm{dist}_M(n,m_\tau)>D\,k^{\varepsilon-1/2}$, where $m_\tau=:\phi^M_{-\tau}(m)$
and $\mathrm{dist}_M$ is the Riemannian distance on $M$.
Similarly, if we set $p=:\pi(z)$ and $p_\tau=:\phi^M_{-\tau}(p)$ then for $z\in W'_k$ we have
$\mathrm{dist}_M(m,p)<(2D/3C)\cdot k^{\varepsilon -1/2}$, and therefore
$\mathrm{dist}_M(m_\tau,p_\tau)<(2D/3)\cdot k^{\varepsilon -1/2}$.

Therefore, for every $z\in W'_k$ and $\vartheta\in (-\pi,\pi)$,
we have
\begin{eqnarray*}
\lefteqn{\mathrm{dist}_X\big(r_\vartheta(z_\tau),y\big)\ge\mathrm{dist}_M\big(p_\tau,n\big)}\\
&\ge&\mathrm{dist}_M\big(n,m_\tau\big)-\mathrm{dist}_M\big(m_\tau,p_\tau\big)>D\,k^{\varepsilon -1/2}
-\frac{2D}{3}\,k^{\varepsilon -1/2}
=\frac D3\,k^{\varepsilon -1/2}.
\end{eqnarray*}

We can now argue as in the proof of Lemma \ref{lem:reduction-k}, and conclude that
$$
\big |\partial _u\Psi_1\big|=\big|\psi \big(r_\vartheta(z_\tau),y\big)\big|\ge \Im \psi \big(r_\vartheta(z_\tau),y\big)\ge C_2\,k^{2\varepsilon-1}.
$$
Using this time using integration by parts in $du$,
we conclude that the second summand on the right hand side of (\ref{eqn:composizione-integrata-partizionata-k})
is also  $O\left(k^{-\infty}\right)$ as $k\rightarrow +\infty$, if
$\mathrm{dist}_X\left(y,S^1\cdot x_\tau\right)>D\,k^{\varepsilon -1/2}$.
\end{proof}

Hence the left hand side of (\ref{eqn:composizione-integrata-partizionata-k})
is $O\left(k^{-\infty}\right)$ for $k\rightarrow +\infty$,
uniformly for $\mathrm{dist}_X\left(y,S^1\cdot x_\tau\right)>D\,k^{\varepsilon-1/2}$.
This completes the proof of Theorem \ref{thm:rapid-decay}.

\section{Proof of Theorem \ref{thm:concentration-rate}.}

Let us set $\varepsilon=1/9$ in the previous construction (this is just to fix ideas).
In view of (\ref{eqn:composizione-integrata-partizionata-k}) and Lemma \ref{lem:reduction-k-1},
writing $z=x+\left(\theta,\mathbf{v}\right)$ in Heisenberg local coordinates we have
\begin{eqnarray}
\label{eqn:composizione-integrata-partizionata-k-1}
\lefteqn{U_{\tau,k}\left(x+\frac{\mathbf{u}}{\sqrt{k}},x_\tau+\frac{\mathbf{w}}{\sqrt{k}}\right)\sim}\\
&&\frac{k^2}{2\pi}\,
\int_{1/E}^{E}\int_{1/E}^{E}\int_{-\pi}^\pi\int_{-\pi}^\pi\int_{\mathbb{C}^\mathrm{d}}e^{i\,k\Psi_1'}\,\gamma_k(z)\,
\mathcal{A}_1\,\mathcal{V}(\theta,\mathbf{v})\,dt\,du\,d\vartheta\,d\theta\,d\mathbf{v},\nonumber
\end{eqnarray}
where, recalling  (\ref{eqn:def-of-phase}),
\begin{eqnarray}
\label{eqn:psi-1}
\Psi_1'&=:&t\,\psi\left(x+\frac{\mathbf{u}}{\sqrt{k}},x+\left(\theta,\mathbf{v}\right)\right)\\
&&+u\,\psi\left(\phi^X_{-\tau}\big(x+(\vartheta+\theta,\mathbf{v}\big)\big),x_\tau+\frac{\mathbf{w}}{\sqrt{k}}\right)
-\vartheta\nonumber
\end{eqnarray}
and integration in $d\mathbf{v}$ is over a ball centered at the origin and radius $O(\epsilon)$.

In particular, since Heisenberg local coordinates are isometric at the origin,
again by Corollary 1.3 of \cite{bs} for sufficiently small $\epsilon$ we have
$$
\big|\partial _t\Psi_1\big|=\left|\psi\left(x+\frac{\mathbf{u}}{\sqrt{k}},x+\left(\theta,\mathbf{v}\right)\right)\right|
\ge a\,|\theta|^2,
$$
for some constant $a>0$. Integrating by parts in $dt$, as in Lemma \ref{lem:reduction-k}, we conclude that only
a small neighborhood of the origin in $(-\pi,\pi)$, say $(-\epsilon/2,\epsilon/2)$, gives a non-negligible contribution to the asymptotics.

Furthermore, by Lemma 3.2 of \cite{pao-ltfII} we have
\begin{eqnarray}
\label{eqn:azione-in-coordinate}
\phi^X_{-\tau}\left(x+\left(\vartheta+\theta,\mathbf{v}\right)\right)=
x_\tau+\Big(\vartheta+\theta+R_3\left(\mathbf{v}\right),A\mathbf{v}+
R_2\left(\mathbf{v}\right)\Big),
\end{eqnarray}
where $R_j$ denotes a generic smooth function on an Euclidean space vanishing to $j$-th order
at the origin (that is, $R_j(\mathbf{s})=O\left(\|\mathbf{s}\|^j\right)$ for $\mathbf{s}\sim \mathbf{0}$).
Therefore, if $\epsilon$ is small, $|\theta|<\epsilon/2$ and $|\vartheta|\ge \epsilon$ then
\begin{eqnarray*}
\big|\partial _u\Psi_1\big|&=&
\left|\psi\left(\phi^X_{-\tau}\big(x+(\vartheta+\theta,\mathbf{v}\big)\big),x_\tau+\frac{\mathbf{w}}{\sqrt{k}}\right)\right|\\
&\ge& a\,\big(\vartheta+\theta+R_3\left(\mathbf{v}\right)\big)^2\ge a'\,\epsilon^2,
\end{eqnarray*}
for some constant $a'>0$. Thus integration by parts in $du$ implies that the contribution to the asymptotics
from the locus where $|\theta|<\epsilon/2$ and $|\vartheta|<\epsilon$ is also $O\left(k^{-\infty}\right)$.

We can thus write
\begin{eqnarray}
\label{eqn:composizione-integrata-partizionata-k-2}
\lefteqn{U_{\tau,k}\left(x+\frac{\mathbf{u}}{\sqrt{k}},x_\tau+\frac{\mathbf{w}}{\sqrt{k}}\right)\sim}\\
&&\frac{k^2}{2\pi}\,
\int_{1/E}^{E}\int_{1/E}^{E}\int_{-\epsilon}^\epsilon\int_{-\epsilon}^\epsilon\int_{\mathbb{C}^\mathrm{d}}e^{i\,k\Psi_1'}\,\gamma_k(z)\,
\,\mathcal{A}_1'\,\mathcal{V}(\theta,\mathbf{v})\,dt\,du\,d\vartheta\,d\theta\,d\mathbf{v},\nonumber
\end{eqnarray}
where $\mathcal{A}_1'=:\varrho(\theta,\vartheta)\,\mathcal{A}_1$, and $\varrho(\theta,\vartheta)$
is an appropriate bump function on
$\mathbb{R}^2$, $\equiv 1$ near the origin and supported on a ball of radius $O(\epsilon)$.

Let us now operate the rescaling $\mathbf{v}\mapsto \mathbf{v}/\sqrt{k}$; we get
\begin{eqnarray}
\label{eqn:composizione-integrata-riscalata}
\lefteqn{U_{\tau,k}\left(x+\frac{\mathbf{u}}{\sqrt{k}},x_\tau+\frac{\mathbf{w}}{\sqrt{k}}\right)}\\
&\sim&\frac{k^{2-\mathrm{d}}}{2\pi}\,\int_{\mathbb{C}^\mathrm{d}}\left[
\int_{1/E}^{E}\int_{1/E}^{E}\int_{-\epsilon}^\epsilon\int_{-\epsilon}^\epsilon
e^{i\,k\Psi_{2}}\,\mathcal{A}_2\cdot
\mathcal{V}\left(\theta,\frac{\mathbf{v}}{\sqrt{k}}\right)\,dt\,du\,d\vartheta\,d\theta\right]\,d\mathbf{v}.
\nonumber
\end{eqnarray}
where $\Psi_2=\Psi_1'(\mathbf{u},\mathbf{w},\mathbf{v},\theta)$ is $\Psi_1'$ expressed in rescaled Heisenberg
coordinates, and similarly for
$\mathcal{A}_2=:\gamma\left(k^{-1/9}\mathbf{v}\right)\,\mathcal{A}_1'$ (dependence on $\tau$ and $k$ is omitted).
Integration in $d\mathbf{v}$ is now over a ball centered at the origin and of radius $O\left(k^{1/9}\right)$ in $\mathbb{C}^\mathrm{d}$.

Thus, by (\ref{eqn:psi-1}),
\begin{eqnarray}
\label{eqn:psi-2}
\lefteqn{\Psi_2=t\,\psi\left(x+\frac{\mathbf{u}}{\sqrt{k}},x+\left(\theta,\frac{\mathbf{v}}{\sqrt{k}}\right)\right)}\\
&&+u\,\psi\left(\phi^X_{-\tau}\left(x+\left(\vartheta+\theta,\frac{\mathbf{v}}{\sqrt{k}}\right)\right),
x_{\tau}+\frac{\mathbf{w}}{\sqrt{k}}\right)-\vartheta.\nonumber
\end{eqnarray}

Let $R_j$ denote a generic smooth function on an Euclidean space vanishing to $j$-th order
at the origin (that is, $R_j(\mathbf{s})=O\left(\|\mathbf{s}\|^j\right)$ for $\mathbf{s}\sim \mathbf{0}$).
By the discussion in \S 3 of \cite{sz}, we have
\begin{eqnarray}
\label{eqn:expansion-first-summand}
\lefteqn{t\,\psi\left(x+\frac{\mathbf{u}}{\sqrt{k}},x+\left(\theta,\frac{\mathbf{v}}{\sqrt{k}}\right)\right)}\\
&=&it\,\left[1-e^{-i\theta}\right]-\frac{it}{k}\,\psi_2(\mathbf{u},\mathbf{v})\,e^{-i\theta}+
t\,R_3\left(\frac{\mathbf{u}}{\sqrt{k}},\frac{\mathbf{v}}{\sqrt{k}}\right)\,e^{-i\theta},\nonumber
\end{eqnarray}
where
\begin{equation}
\label{eqn:def-psi2}
\psi_2(\mathbf{u},\mathbf{v})=:-i\,\omega_0(\mathbf{u},\mathbf{v})-\frac 12\,\|\mathbf{u}-\mathbf{v}\|^2.
\end{equation}

Again by Lemma 3.2 of \cite{pao-ltfII}, we have
\begin{eqnarray}
\label{eqn:azione-in-coordinate}
\lefteqn{\phi^X_{-\tau}\left(x+\left(\vartheta+\theta,\frac{\mathbf{v}}{\sqrt{k}}\right)\right)}\\
&=&x_\tau+\left(\vartheta+\theta+R_3\left(\frac{\mathbf{v}}{\sqrt{k}}\right),\frac{A\mathbf{v}}{\sqrt{k}}+
R_2\left(\frac{\mathbf{v}}{\sqrt{k}}\right)\right).\nonumber
\end{eqnarray}
It follows that
\begin{eqnarray}
\label{eqn:expansion-second-summand}
\lefteqn{u\,\psi\left(\phi^X_{-\tau}\left(x+\left(\vartheta+\theta,\frac{\mathbf{v}}{\sqrt{k}}\right)\right),
x+\frac{\mathbf{w}}{\sqrt{k}}\right)}\\
&=&iu\,\left[1-e^{i(\theta+\vartheta)}\right]-\frac{iu}{k}\,\psi_2(A\mathbf{v},\mathbf{w})\,e^{i(\theta+\vartheta)}+
u\,R_3\left(\frac{A\mathbf{v}}{\sqrt{k}},\frac{\mathbf{w}}{\sqrt{k}}\right)\,e^{i(\theta+\vartheta)},\nonumber
\end{eqnarray}
Inserting (\ref{eqn:expansion-first-summand}) and (\ref{eqn:expansion-second-summand}) in
(\ref{eqn:psi-2}), we can rewrite (\ref{eqn:composizione-integrata-riscalata}) as follows:
\begin{eqnarray}
\label{eqn:composizione-integrata-riscalata-espansa}
\lefteqn{U_{\tau,k}\left(x+\frac{\mathbf{u}}{\sqrt{k}},x_\tau+\frac{\mathbf{w}}{\sqrt{k}}\right)}\\
&\sim&\frac{k^{2-\mathrm{d}}}{2\pi}\,\int_{\mathbb{C}^\mathrm{d}}\left[
\int_{1/E}^{E}\int_{1/E}^{E}\int_{-\epsilon}^\epsilon\int_{-\epsilon}^\epsilon e^{i\,k\Psi}\,\mathcal{B}_k\cdot
\mathcal{V}\left(\theta,\frac{\mathbf{v}}{\sqrt{k}}\right)\,dt\,du\,d\vartheta\,d\theta\right]\,d\mathbf{v},
\nonumber
\end{eqnarray}
where
\begin{equation}
\label{eqn:def-di-Psi}
\Psi=:it\,\left[1-e^{-i\theta}\right]+iu\,\left[1-e^{i(\theta+\vartheta)}\right]-\vartheta,
\end{equation}
\begin{eqnarray}
\label{eqn:defn-ampiezza}
\mathcal{B}_k&=:&\exp\left(t\,\psi_2\big(\mathbf{u},\mathbf{v}\big)\,e^{-i\theta}
+u\,\psi_2\big(A\mathbf{v},\mathbf{w}\big)\,e^{i(\theta+\vartheta)}\right)\\
&&\exp\left(ik\,t\,R_3\left(\frac{\mathbf{u}}{\sqrt{k}},\frac{\mathbf{v}}{\sqrt{k}}\right)\,e^{-i\theta}+
ik\,u\,R_3\left(\frac{A\mathbf{v}}{\sqrt{k}},\frac{\mathbf{w}}{\sqrt{k}}\right)\,e^{i(\theta+\vartheta)}\right)\cdot
\mathcal{A}_2.\nonumber
\end{eqnarray}

\begin{lem}
\label{lem:real-negative-part}
There exists $a=a_\tau>0$ such that for any $(\mathbf{u},\mathbf{w})\in N_{\tau,m}$ and
$\mathbf{v}\in T_mM$ we have
$$
\Re\left(t\,\psi_2\big(\mathbf{u},\mathbf{v}\big)\,e^{-i\theta}
+u\,\psi_2\big(A\mathbf{v},\mathbf{w}\big)\,e^{i(\theta+\vartheta)}\right)\le
-a\,\left(\|\mathbf{u}\|^2+\|\mathbf{w}\|^2+\|\mathbf{v}\|^2\right),
$$
\end{lem}

\begin{proof}
The linear map $N_{\tau,m}\times T_mM\rightarrow T_mM\times T_{m_\tau}M$ given in local coordinates
by $(\mathbf{u},\mathbf{w},\mathbf{v})\mapsto (\mathbf{u}-\mathbf{v},A\mathbf{v}-\mathbf{w})$ is injective
by assumption. Therefore, $\|\mathbf{u}-\mathbf{v}\|^2+\|A\mathbf{v}-\mathbf{w}\|^2\ge
C\left(\|\mathbf{u}\|^2+\|\mathbf{w}\|^2+\|\mathbf{v}\|^2\right)$ for some $C>0$. The statement
follows from the definition of $\psi_2$ and the fact that $|\theta|,|\vartheta|<\epsilon$.ù
\end{proof}

The second exponent on the right hand side of (\ref{eqn:defn-ampiezza}), on the other hand, is bounded
for $\|\mathbf{u}\|,\,\|\mathbf{w}\|,\,\|\mathbf{v}\|=O\left(k^{1/9}\right)$. Taylor expanding the exponent
at the origin yields an asymptotic expansion of the corresponding exponential in descending powers of
$k^{-1/2}$, which may be incorporated into the amplitude.

We are then in a position to apply the stationary phase Lemma, regarding the inner integral in
(\ref{eqn:composizione-integrata-riscalata-espansa}) as an oscillatory integral, with
phase $\Psi=\Psi (t,\theta,u,\vartheta)$ having non-negative imaginary part. A straightforward computation yields:

\begin{lem}
$\Psi$ has the unique stationary point
$$
(t_0,\theta_0,u_0,\vartheta_0)=(1,0,1,0).
$$
Furthermore, the Hessian matrix there is
$$
H(\Psi)_0=\left(
            \begin{array}{cccc}
              0 & -1 & 0 & 0 \\
              -1 & 2i & 1 & i \\
              0 & 1 & 0 & 1 \\
              0 & i & 1 & i \\
            \end{array}
          \right).
          $$
          In particular, $\det \big(H(\Psi)_0\big)=1$ and the stationary point is non-degenerate.
\end{lem}

In addition, $H(\Psi)_0=H(1)$, where for $0\le s\le 1$ we set
$$
H(s)=:\left(
            \begin{array}{cccc}
              0 & -1 & 0 & 0 \\
              -1 & 2s i & 1 & s i \\
              0 & 1 & 0 & 1 \\
              0 & s i & 1 & s i \\
            \end{array}
          \right).
          $$
          We have $H(s)=1$ for every $s$, and $H(0)$ is real and symmetric with vanishing
signature. Therefore,
\begin{equation}
\label{eqn:sqrt-factor}
\sqrt{\det \left(\frac{k\,H(\Psi)_0}{2\pi i}\right)}=\left(\frac{k}{2\pi}\right)^2.
\end{equation}
Applying the stationary phase Lemma, we get for the inner integral in (\ref{eqn:composizione-integrata-riscalata-espansa})
an asymptotic expansion in descending powers of $k^{-1/2}$, and it follows from
Lemma \ref{lem:real-negative-part} and the bound on the $N$-th step remainder
that this expansion may be integrated term by term in
$d\mathbf{v}$, yielding an asymptotic expansion for (\ref{eqn:composizione-integrata-riscalata-espansa}). Given (\ref{eqn:amplitudes-leading-term})
and (\ref{eqn:sqrt-factor}),
and since
$\mathcal{V}(\theta,\mathbf{0})=1/(2\pi)$, the leading term is
\begin{equation}
\label{eqn:leading-term}
\frac{k^{\mathrm{d}}}{\pi^{2\mathrm{d}}}\,\varrho_\tau(m)\,
\int_{\mathbb{C}^\mathrm{d}}e^{\psi_2(\mathbf{u},\mathbf{v})+\psi_2(A\mathbf{v},\mathbf{w})}\,d\mathbf{v}.
\end{equation}
With the change of variable $\mathbf{v}=\mathbf{r}+\mathbf{u}$, we get
\begin{eqnarray}
\label{eqn:1st-manipulation}
\lefteqn{\psi_2(\mathbf{u},\mathbf{v})+\psi_2(A\mathbf{v},\mathbf{w})}\\
&=&\psi_2(A\mathbf{u},\mathbf{w})-i\,\omega_0\left(A^{-1}L(\mathbf{u},\mathbf{w}),\mathbf{r}\right)
-\mathbf{r}^t\,A^t\,L(\mathbf{u},\mathbf{w})-\frac 12\,\mathbf{r}^t\,Q\,\mathbf{r},\nonumber
\end{eqnarray}
where $L=L_A$ and $Q=Q_A$ are as in Definition \ref{defn:key-bilinear-form}. With the further replacement
$\mathbf{r}=\mathbf{s}-Q^{-1} A L(\mathbf{u},\mathbf{w})$ in (\ref{eqn:1st-manipulation}), we get
\begin{eqnarray}
\label{eqn:2nd-manipulation}
\psi_2(\mathbf{u},\mathbf{v})+\psi_2(A\mathbf{v},\mathbf{w})
&=&\Gamma(\mathbf{u},\mathbf{w})\nonumber\\
 &&-i\,\mathbf{s}^t\,J_0\,A^{-1}\,L(\mathbf{u},\mathbf{w})-\frac 12\,\mathbf{s}^tQ\,\mathbf{s},
\end{eqnarray}
where
\begin{eqnarray}
\label{eqn:Gamma}
\Gamma(\mathbf{u},\mathbf{w}) &=:&\psi_2(A\mathbf{u},\mathbf{w})+i\,\omega_0\left(A^{-1}L(\mathbf{u},\mathbf{w}),Q^{-1}A^tL(\mathbf{u},\mathbf{w})\right)
\nonumber
\\
&&
+\frac 12\,L(\mathbf{u},\mathbf{w})^t\,A\,Q^{-1}\,A^tL(\mathbf{u},\mathbf{w}).
\end{eqnarray}
Therefore, the leading term (\ref{eqn:leading-term}) is
\begin{equation}
\label{eqn:leading-term-1}
\frac {k^{\mathrm{d}}}{\pi^{2\mathrm{d}}}\,\varrho_\tau(m)\,e^{\Gamma(\mathbf{u},\mathbf{w})}\,
\int_{\mathbb{C}^\mathrm{d}}e^{-i\,\mathbf{s}^t\,J_0\,A^{-1}\,L(\mathbf{u},\mathbf{w})-\frac 12\,\mathbf{s}^tQ\,\mathbf{s}}\,d\mathbf{s}.
\end{equation}
Let us set
\begin{equation}
\label{eqn:F-G}
F(\mathbf{u},\mathbf{w})=:-J_0A^{-1}L(\mathbf{u},\mathbf{w})=-A^tJ_0L(\mathbf{u},\mathbf{w}),\,\,\,\,
G(\mathbf{u},\mathbf{w})=:A^t\,L(\mathbf{u},\mathbf{w}).
\end{equation}
Then with some manipulations (\ref{eqn:leading-term-1}) is
\begin{eqnarray}
\label{eqn:leading-term-2}
\lefteqn{\left(\frac {k}{\pi}\right)^{\mathrm{d}}\,\varrho_\tau(m)\cdot
\frac{2^\mathrm{d}}{\sqrt{\det (Q)}}\cdot e^{\Gamma(\mathbf{u},\mathbf{w})-\frac 12\,F(\mathbf{u},\mathbf{w})^tQ^{-1}F(\mathbf{u},\mathbf{w})}}\nonumber\\
&=&\left(\frac {k}{\pi}\right)^{\mathrm{d}}\,\varrho_\tau(m)\cdot
\frac{2^\mathrm{d}}{\sqrt{\det (Q)}}\cdot e^{S(\mathbf{u},\mathbf{w})},
\end{eqnarray}
where
\begin{eqnarray}
\label{eqn:Scal}
S(\mathbf{u},\mathbf{w})&=:&\psi_2(A\mathbf{u},\mathbf{w})-i\,G(\mathbf{u},\mathbf{w})\,Q^{-1}F(\mathbf{u},\mathbf{w})\\
&&+\frac 12\,G(\mathbf{u},\mathbf{w})^t\,Q^{-1}\,G(\mathbf{u},\mathbf{w})-
\frac 12\,F(\mathbf{u},\mathbf{w})^tQ^{-1}\,F(\mathbf{u},\mathbf{w}).\nonumber
\end{eqnarray}

\begin{lem}
\label{lem:leading-term}
If $\mathcal{P}=\mathcal{P}_A$ and $\mathcal{R}=\mathcal{R}_A$ are as in Definition \ref{defn:key-bilinear-form},
then
$$
S(\mathbf{u},\mathbf{w})=-L(\mathbf{u},\mathbf{w})^t\,\left(\mathcal{P}+\frac i2\,\mathcal{R}\right)\,L(\mathbf{u},\mathbf{w})
-i\,\omega_0(A\mathbf{u},\mathbf{w}).
$$
\end{lem}

\begin{proof}
By (\ref{eqn:def-psi2}),
$\psi_2(A\mathbf{u},\mathbf{w})=-i\,\omega_0(A\mathbf{u},\mathbf{w})-(1/2)\,\|L(\mathbf{u},\mathbf{w})\|^2$.
Using this and (\ref{eqn:F-G}) in (\ref{eqn:Scal}), we get
\begin{eqnarray}\label{eqn:Scal-1}
S(\mathbf{u},\mathbf{w})&=&
-\frac 12\,L(\mathbf{u},\mathbf{w})^t\,\left[I-J_0AQ^{-1}A^tJ_0-AQ^{-1}A^t\right]\,L(\mathbf{u},\mathbf{w})\nonumber
\\
&&+i\,L(\mathbf{u},\mathbf{w})^t\,AQ^{-1}A^tJ_0L(\mathbf{u},\mathbf{w})-i\,\omega_0(A\mathbf{u},\mathbf{w}).
\end{eqnarray}
Writing $A=OP$, we get (see Lemma 2.1 of \cite{p-ltfII})
\begin{equation}
\label{eqn:relazione-matriciale}
I-J_0AQ^{-1}A^tJ_0-AQ^{-1}A^t=2\,OQ^{-1}O^t=2\,\mathcal{P},
\end{equation}
and on the other hand
$AQ^{-1}A^tJ_0=OP^2Q^{-1}J_0O^t$; on the other hand, since $P$ is symplectic and symmetric, $\left(I+P^{2}\right)\,J_0=J_0\,\left(I+P^{-2}\right)$.
Therefore,
\begin{eqnarray}
\label{eqn:relazione-matriciale1}
AQ^{-1}A^tJ_0+\left(AQ^{-1}A^tJ_0\right)^t
&=&O\left[P^2-I\right]Q^{-1}J_0 O^t\nonumber\\
&=&-\mathcal{R}.
\end{eqnarray}
The statement follows by inserting (\ref{eqn:relazione-matriciale}) and (\ref{eqn:relazione-matriciale1})
in (\ref{eqn:Scal-1}).
\end{proof}

Thus $S(\mathbf{u},\mathbf{w})=\mathcal{S}_A(\mathbf{u},\mathbf{w})$, and this proves that the leading term
of the asymptotic expansion  is as claimed in the statement of the Theorem.

By the same arguments, the general lower order
term in the expansion has the form
\begin{eqnarray}
\label{eqn:general-term-expansion-1}
\lefteqn{k^{\mathrm{d}-j/2}\,\varrho_\tau(m)\,
\int_{\mathbb{C}^\mathrm{d}}P_j(\mathbf{u},\mathbf{w},\mathbf{v})\,
e^{\psi_2(\mathbf{u},\mathbf{v})+\psi_2(A\mathbf{v},\mathbf{w})}\,d\mathbf{v}}\\
&=&k^{\mathrm{d}-j/2}\,\varrho_\tau(m)\,e^{\Gamma(\mathbf{u},\mathbf{w})}\,
\int_{\mathbb{C}^\mathrm{d}}e^{-i\,\mathbf{s}^t\,J_0\,A^{-1}\,L(\mathbf{u},\mathbf{w})}
\widetilde{P}_j(\mathbf{u},\mathbf{w},D)\left(e^{-\frac 12\,\mathbf{s}^tQ\,\mathbf{s}}\right)\,d\mathbf{s},
\nonumber
\end{eqnarray}
where $j$ is a positive integer, $P_j$ a polynomial, and $\widetilde{P}_j$
a differential operator with coefficients depending polynomially on $(\mathbf{u},\mathbf{w})$.
This may be rewritten
$$
k^{\mathrm{d}-j/2}\,\,\varrho_\tau(m)e^{\mathcal{S}_{\tau,m}(\mathbf{u},\mathbf{w})}\cdot
a_j(m,\tau,\mathbf{u},\mathbf{w}),
$$
for a certain polynomial $a_j$, depending smoothly on $m$ and $\tau$.

Let us now consider the last claim of the Theorem. Since on the one hand
the asymptotic expansions for the amplitudes
in (\ref{eqn:asymp-exp-amplitues}) go down by integer steps, and on the other the inner integral in
(\ref{eqn:composizione-integrata-riscalata-espansa}) is oscillatory in $k$, the appearance of
half-integer powers of $k$ is the asymptotic expansion of Theorem \ref{thm:concentration-rate}
originates solely from Taylor expanding the amplitude in (\ref{eqn:defn-ampiezza})
in the rescaled arguments $\mathbf{u}/\sqrt{k},\,\mathbf{w}/\sqrt{k},\,\mathbf{v}/\sqrt{k}$.
Therefore, the general term (\ref{eqn:general-term-expansion-1}) of the expansion is actually
a sum of terms of the form
\begin{eqnarray}
\label{eqn:general-term-expansion-2}
k^{\mathrm{r}-|\ell|/2}\,\varrho_\tau(m)\,
\int_{\mathbb{C}^\mathrm{d}}P_{\ell}(\mathbf{u},\mathbf{w},\mathbf{v})\,
e^{\psi_2(\mathbf{u},\mathbf{v})+\psi_2(A\mathbf{v},\mathbf{w})}\,d\mathbf{v},
\end{eqnarray}
where $r$ is an integer, $P_{\ell}(\mathbf{u},\mathbf{w},\mathbf{v})$ is a polyhomogenous polynonomial
in $(\mathbf{u},\mathbf{w},\mathbf{v})$, of polydegree $\ell=(l_\mathbf{u},l_\mathbf{w},l_\mathbf{v})$,
and $|\ell|=l_\mathbf{u}+l_\mathbf{w}+l_\mathbf{v}$; the
coefficients are smooth in $m$.

By the previous passages, involving the change of variable
$\mathbf{v}=\mathbf{s}-Q^{-1} A L(\mathbf{u},\mathbf{w})+\mathbf{u}$,
the integral in (\ref{eqn:general-term-expansion-2}) may be rewritten as a sum
of terms of the form
\begin{eqnarray}
\label{eqn:general-term-expansion-3}
e^{\Gamma(\mathbf{u},\mathbf{w})}
\,\int_{\mathbb{C}^\mathrm{d}}
e^{-i\,\mathbf{s}^t\,J_0\,A^{-1}\,L(\mathbf{u},\mathbf{w})}\,R_{\ell''}(\mathbf{u},\mathbf{w},\mathbf{s})
e^{-\frac 12\,\mathbf{s}^tQ\,\mathbf{s}}\,d\mathbf{s},
\end{eqnarray}
where again $R_{\ell'}$ is polyhomogenous, of polydegree $\ell'$ with $|\ell'|=|\ell|$.

In turn, (\ref{eqn:general-term-expansion-3}) splits as a sum of terms of the form
\begin{eqnarray}
\label{eqn:general-term-expansion-4}
\lefteqn{e^{\Gamma(\mathbf{u},\mathbf{w})}
\,\int_{\mathbb{C}^\mathrm{d}}
e^{-i\,\mathbf{s}^t\,J_0\,A^{-1}\,L(\mathbf{u},\mathbf{w})}\,\widehat{R}_{\ell'}(\mathbf{u},\mathbf{w},D_\mathbf{s})
\left(e^{-\frac 12\,\mathbf{s}^tQ\,\mathbf{s}}\right)\,d\mathbf{s}}\nonumber
\\
&=&c\,\widehat{R}_{\ell'}\big(\mathbf{u},\mathbf{w},F(\mathbf{u},\mathbf{w})\big)\,e^{\mathcal{S}_{\tau,m}(\mathbf{u},\mathbf{w})}
\end{eqnarray}
where now $\widehat{R}_{\ell''}$ is polyhomogenous of polydegree $\ell''=(l'_\mathbf{u},l'_\mathbf{w},2a+l'_\mathbf{s})$
for some integer $a$.

Summing up, the general summand (\ref{eqn:general-term-expansion-2}) splits as a linear combination of terms
of the form
$$
k^{b-|\ell'''|/2}\cdot \widetilde{R}_{\ell'''}(\mathbf{u},\mathbf{w})\,e^{\mathcal{S}_{\tau,m}(\mathbf{u},\mathbf{w})},
$$
where $b$ is an integer, and $\widetilde{R}_{\ell'''}(\mathbf{u},\mathbf{w})$
is polyhomogenous of polydegree $\ell'''=(l_\mathbf{u}''',l_\mathbf{w}''')$. The claim follows.

\section{Proof of Corollary \ref{cor:condition-on-symbol}}

If $U_{\tau,k}$ is unitary for $k\gg 0$, then $U_{\tau,k}\circ U_{\tau,k}^*=\Pi_k$
for $k$ large. In particular, for any $x\in X$ this implies
\begin{equation}
\label{eqn:diagonal-behaviour}
\left(U_{\tau,k}\circ U_{\tau,k}^*\right)(x,x)=\Pi_k(x,x)=\left(\frac k\pi\right)^\mathrm{d}+O\left(k^{\mathrm{d}-1}\right).
\end{equation}
Now we have
\begin{equation}
\label{eqn:composition-integral}
\left(U_{\tau,k}\circ U_{\tau,k}^*\right)(x,x)=\int_XU_{\tau,k}(x,y) U_{\tau,k}^*(y,x)\,d\mu_X(y),
\end{equation}
where $U_{\tau,k}^*(y,x)=\overline{U_{\tau,k}(x,y)}$.

By Theorem \ref{thm:rapid-decay} with $\varepsilon=1/9$,
only a shrinking $S^1$-invariant neighborhood of $x_\tau$, of radius say $O\left(k^{-7/18}\right)$,
contributes non-negligibly to the asymptotics. So introducing Heisenberg local coordinates centered at
$x_\tau$, and with $\gamma_k$ as in (\ref{eqn:composizione-integrata-partizionata-k}),
we can rewrite (\ref{eqn:composition-integral})
as follows:
\begin{eqnarray}
\label{eqn:integration-rescaled}
  \lefteqn{\left(U_{\tau,k}\circ U_{\tau,k}^*\right)(x,x)\sim \int _X U_{\tau,k}(x,y) U_{\tau,k}^*(y,x)\,\gamma_k(y)\,d\mu_X(y) }\\
   &=&k^{-\mathrm{d}}\int _X \left|U_{\tau,k}\left(x,x_\tau+\left(\theta,\frac{\mathbf{v}}{\sqrt{k}}\right) \right) \right|^2 \mathcal{V}\left(\theta,\frac{\mathbf{v}}{\sqrt{k}}\right)\,\gamma\left(k^{-1/9}\mathbf{v}\right)\,d\mathbf{v}\,d\theta . \nonumber\end{eqnarray}
   Using Theorem \ref{thm:concentration-rate}, and recalling that $\mathcal{V}\left(\theta,\mathbf{0}\right)=1/(2\pi)$,
   we get with $Q=Q_A$:
\begin{eqnarray}
\label{eqn:integration-rescaled-1}
\left(U_{\tau,k}\circ U_{\tau,k}^*\right)(x,x)&\sim &
\frac{k^{\mathrm{d}}}{\pi^{2\mathrm{d}}}\left|\varrho_\tau(m)\right|^2\,\frac{2^{2\mathrm{d}}}{\det(Q)}\,
\int_{\mathbb{C}^\mathrm{d}}e^{2\,\Re\big(\mathcal{S}(\mathbf{0},\mathbf{v})\big)}\,d\mathbf{v}\nonumber\\
&&+O\left(k^{\mathrm{d}-1}\right);
\end{eqnarray}
in passing, that the remainder is $O\left(k^{\mathrm{d}-1}\right)$ rather than $O\left(k^{\mathrm{d}-1/2}\right)$
follows directly from the parity Claim in Theorem \ref{thm:concentration-rate}.
In view of Definition \ref{defn:key-bilinear-form},
\begin{equation}
\label{eqn:real-part}
2\,\Re\big(\mathcal{S}(\mathbf{0},\mathbf{v})\big)=-2\,\mathbf{v}^t\,\mathcal{P}\,\mathbf{v}
=-2\,\mathbf{v}^t\,OQ^{-1}O^t\,\mathbf{v}.
\end{equation}
Setting $\mathbf{s}=O^t\,\mathbf{v}$, and then $\mathbf{r}=\mathbf{s}/2$,
the integral in (\ref{eqn:integration-rescaled-1}) is
\begin{eqnarray}
\label{eqn:integral-real-part}
\lefteqn{\int_{\mathbb{C}^\mathrm{d}}e^{-2\,\mathbf{v}^t\,OQ^{-1}O^t\,\mathbf{v}}\,d\mathbf{v}
=\int_{\mathbb{C}^\mathrm{d}}e^{-2\,\mathbf{s}^t\,Q^{-1}\,\mathbf{s}}\,d\mathbf{s}}\\
&=&2^{-2\mathrm{d}}\int_{\mathbb{C}^\mathrm{d}}e^{-\frac 12\,\mathbf{r}^t\,Q^{-1}\,\mathbf{r}}\,d\mathbf{r}=
2^{-2\mathrm{d}}\left(2 \pi \right)^\mathrm{d}\,\sqrt{\det(Q)}.
\nonumber
\end{eqnarray}
Inserting (\ref{eqn:integral-real-part}) in (\ref{eqn:integration-rescaled-1}), we get
\begin{equation}
\label{eqn:leading-order-comp}
\left(U_{\tau,k}\circ U_{\tau,k}^*\right)(x,x)=\left(\frac k\pi\right)^\mathrm{d}\,\left|\varrho_\tau(m)\right|^2\,
\frac{2^{\mathrm{d}}}{\sqrt{\det (Q)}}+O\left(k^{\mathrm{d}-1}\right).
\end{equation}
Comparing (\ref{eqn:leading-order-comp}) with (\ref{eqn:diagonal-behaviour}),
we conclude that $\left|\varrho_\tau(m)\right|=2^{-\mathrm{d}/2}\cdot\det (Q)^{1/4}$ if $U_{\tau,k}$ is unitary for $k\gg 0$.

\section{Proof of Corollary \ref{cor:esiste-unitario}.}

As a preliminary remark, we recall that for any integer $a\ge 0$
a Toeplitz operator $Q$ of degree $-a$ may be written microlocally in the
form
\begin{equation}
\label{eqn:toeplitz-a}
Q\left(y',y''\right)=:\int_0^{+\infty}e^{it\,\psi\left(y',y''\right)}\,q\left(t,y',y''\right)\,dt,
\end{equation}
where the amplitude $q$ is a semiclassical symbol admitting an asymptotic expansion of the form
\begin{equation}
\label{eqn:amplitude-q}
q\left(t,y',y''\right)\sim \sum_{j\ge a}t^{\mathrm{d}-j}\,q_j\left(y',y''\right).
\end{equation}

On the other hand, if $Q$ is $S^1$-invariant then by the discussion in \cite{g-star}
it also admits an asymptotic expansion of the form
\begin{equation}
\label{eqn:toeplitz-b}
Q\sim \sum_{j\ge a}T^{-j}\,\Pi\circ M_{f_j}\circ \Pi,
\end{equation}
where now $f_j\in \mathcal{C}^\infty(M)$ is implicitly pulled-back to $X$,
$M_{f_j}$ is multiplication by $f_j$, and $T$ is a parametrix (in the Toeplitz sense)
of the elliptic first order Toeplitz operator associated to the generator
of the structure circle action. The symbol of $Q$, in particular, is the
function $\sigma(Q):\Sigma\rightarrow \mathbb{C}$ given by
$\sigma(Q)\big(x,r\alpha_x)=r^{-a}\,f_a(m)$, where $m=:\pi(x)$.

When working in Heisenberg local coordinates centered at $x\in X$,
\begin{equation}
\label{eqn:relation-symbols}
q_a(x,x)=\frac{1}{\pi^\mathrm{d}}\,f_a(m).
\end{equation}

Now consider the intrinsically defined asymptotic expansion
\begin{equation}
\label{eqn:asympt-pi}
\Pi_k(x,x)\sim \left(\frac k\pi\right)^{\mathrm{d}}+\sum _{j\ge 1}k^{\mathrm{d}-j}\,a_j(m),
\end{equation}
for certain $a_j\in \mathcal{C}^\infty(M)$.
Let $U_\tau=U_\tau^{[1]}$ be as in (\ref{eqn:defn-U-tau}), for some zeroth order Toeplitz operator
$R_\tau=R_\tau^{[1]}$ with $\varrho_\tau(m)=2^{-\mathrm{d}/2}\,\sqrt{\nu (\tau,m)}$.
The proof of Corollary \ref{cor:condition-on-symbol} implies
\begin{equation}
\label{eqn:asympt-1st-approx}
\left(U_{\tau,k}^{[1]}\circ \big(U_{\tau,k}^{[1]}\big)^*\right)(x,x)\sim \left(\frac k\pi\right)^{\mathrm{d}}+\sum _{j\ge 1}k^{\mathrm{d}-j}\,a_j^{[1]}(m),
\end{equation}
for certain $a_j^{[1]}\in \mathcal{C}^\infty(M)$; that the expansion goes down by integer steps can be seen - for instance -
by using in (\ref{eqn:integration-rescaled})
the parity properties of the $a_j$'s asserted in Theorem \ref{thm:concentration-rate}.

Next let $U_\tau^{[2]}$ be again as in (\ref{eqn:defn-U-tau}), but with
$R_\tau^{[1]}$ replaced by $R_\tau^{[2]}=:R_\tau^{[1]}+\Sigma_\tau^{[1]}$, where
$$
\Sigma_\tau^{[1]}= T^{-1}\,\Pi\circ M_{f_1}\circ \Pi,$$
for a suitable $f_1\in \mathcal{C}^\infty(M\times \mathbb{R})$. Thus $\Sigma_\tau^{[1]}$
is a Toeplitz operator of degree $-1$, hence microlocally of the form
\begin{equation}
\label{eqn:S-remained-2nd-step}
\Sigma_\tau^{[1]}\left(y',y''\right)=:\int_0^{+\infty}e^{it\,\psi\left(y',y''\right)}\,\sigma_\tau^{[1]}\left(t,y',y''\right)\,dt,
\end{equation}
with
$$
\sigma_\tau^{[1]}\left(t,y',y''\right)\sim \sum_{j\ge 1}t^{\mathrm{d}-j}\,\sigma_{\tau j}^{[1]}\left(y',y''\right),
$$
and $\sigma_{\tau 1}^{[1]}(x,x)=f_1(m,\tau)/\pi^\mathrm{d}$.
Applying the stationary phase argument in the proof of Theorem \ref{thm:concentration-rate}, and arguing as for
(\ref{eqn:leading-order-comp}), we get
\begin{eqnarray}
\label{eqn:modified-expansion}
\lefteqn{\left(U_{\tau,k}^{[2]}\circ \big(U_{\tau,k}^{[2]}\big)^*\right)(x,x)}\\
&\sim& \left(U_{\tau,k}^{[1]}\circ \big(U_{\tau,k}^{[1]}\big)^*\right)(x,x)
+\frac{k^{\mathrm{d}-1}}{\pi^{\mathrm{d}}}\varrho_\tau(m)\cdot 2\,\Re\big(f_1(m,\tau)\big)\,
\frac{2^{\mathrm{d}}}{\sqrt{\det (Q)}}\nonumber\\
&&+O\left(k^{\mathrm{d}-2}\right).    \nonumber
\end{eqnarray}
It is then clear that $f_1:M\rightarrow \mathbb{R}$ may be chosen uniquely so that (\ref{eqn:modified-expansion})
agrees with (\ref{eqn:asympt-pi}) up to $O\left(k^{\mathrm{d}-2}\right)$.

Proceeding inductively, there are unique real $f_j\in \mathcal{C}^\infty(M\times \mathbb{R})$, such that
if $R_\tau^{[\infty]}\sim R_\tau+\sum_{j\ge 1}T^{-j}\,\Pi\circ M_{f_j}\circ \Pi$ and $U_\tau^{[\infty]}$
is as in (\ref{eqn:defn-U-tau}), with $R_\tau^{[\infty]}$ in place of $R_\tau$, then
\begin{equation}
\label{eqn:modified-expansion-infty}
\left(U_{\tau,k}^{[\infty]}\circ \big(U_{\tau,k}^{[\infty]}\big)^*\right)(x,x)\sim \Pi_k(x,x),
\end{equation}
hence $U_{\tau,k}^{(\infty)}\circ \big(U_{\tau,k}^{(\infty)}\big)^*=\Pi_k+O\left(k^{-\infty}\right)$.
Therefore, $\big(U_{\tau,k}^{(\infty)}\big)^*\circ U_{\tau,k}^{(\infty)}\ge 0$ and
$\big(U_{\tau,k}^{(\infty)}\big)^*\circ U_{\tau,k}^{(\infty)}= \left(\big(U_{\tau,k}^{(\infty)}\big)^*\circ U_{\tau,k}^{(\infty)}\right)^2
+O\left(k^{-\infty}\right)$; working in an orthonormal basis of eigenvectors, we conclude that
$\big(U_{\tau,k}^{(\infty)}\big)^*\circ U_{\tau,k}^{(\infty)}=\Pi_k+O\left(k^{-\infty}\right)$
as well.

\section{Proof of Proposition \ref{prop:1-parameter-group}}

The proof is an adaptation of the one for Theorem \ref{thm:concentration-rate},
so we'll be rather sketchy.
Working at a fixed $\tau_0$, let us define operators
$\widetilde{U}_{\tau,k}=:U_{\tau_0+\tau/\sqrt{k},k}$, so that
\begin{equation}
\label{eqn:rescaled-derivative}
\left.\frac{dU_{\tau,k}}{d\tau}\right|_{\tau_0}=
\sqrt{k}\cdot\left.\frac{d\widetilde{U}_{\tau,k}}{d\tau}\right|_{0}.
\end{equation}
Arguing as in the proof of Theorem \ref{thm:concentration-rate}, we have in place of (\ref{eqn:composizione-integrata-riscalata})
\begin{eqnarray}
\label{eqn:composizione-integrata-riscalata-bis}
\lefteqn{\widetilde{U}_{\tau,k}\left(x+\frac{\mathbf{u}}{\sqrt{k}},x_{\tau_0}+\frac{\mathbf{w}}{\sqrt{k}}\right)}\\
&\sim&\frac{k^{2-\mathrm{d}}}{2\pi}\,\int_{\mathbb{C}^\mathrm{d}}\left[
\int_{1/E}^{E}\int_{1/E}^{E}\int_{-\epsilon}^\epsilon\int_{-\epsilon}^\epsilon
e^{i\,k\widetilde{\Psi}_{2}}\,\widetilde{\mathcal{A}}_2\cdot
\mathcal{V}\left(\theta,\frac{\mathbf{v}}{\sqrt{k}}\right)\,dt\,du\,d\vartheta\,d\theta\right]\,d\mathbf{v},
\nonumber
\end{eqnarray}
where now
\begin{eqnarray}
\label{eqn:psi-riscalata}
\lefteqn{\widetilde{\Psi}_2=t\,\psi\left(x+\frac{\mathbf{u}}{\sqrt{k}},x+\left(\theta,\frac{\mathbf{v}}{\sqrt{k}}\right)\right)}\\
&&+u\,\psi\left(\phi^X_{-(\tau_0+\tau/\sqrt{k})}\left(x+\left(\vartheta+\theta,\frac{\mathbf{v}}{\sqrt{k}}\right)\right),
x_{\tau_0}+\frac{\mathbf{w}}{\sqrt{k}}\right)-\vartheta,\nonumber
\end{eqnarray}
and the amplitude $\widetilde{\mathcal{A}}_2$
is similarly redefined.
We may assume without loss that integration in $d\vartheta\,d\theta$ is compactly supported near the origin.

\begin{lem}
Choose $C_0>0$. There exist constants $C_1,C_2>0$ such that, uniformly in
$|\tau|<C_0$, the contribution to the asymptotics of (\ref{eqn:composizione-integrata-riscalata-bis})
of the locus where $|\theta|>C_1\,k^{-7/18}$, $|\vartheta|>C_2\,k^{-7/18}$ is $O\left(k^{-\infty}\right)$.
\end{lem}

\begin{proof}
Let $C_1>0$ be arbitrary, and suppose $|\theta|>C_1\,k^{-7/18}$.
Then
\begin{equation}
\label{eqn:bound-on-distance}
\mathrm{dist}_X\left(x+\frac{\mathbf{u}}{\sqrt{k}},x+\left(\theta,\frac{\mathbf{v}}{\sqrt{k}}\right)\right)>
\frac{C_1}{2}\,\,k^{-7/18},
\end{equation}
and therefore
\begin{equation}
\label{eqn:bound-on-psi}
\big|\partial_t\widetilde{\Psi}_2\big|=
\left|\psi\left(x+\frac{\mathbf{u}}{\sqrt{k}},x+\left(\theta,\frac{\mathbf{v}}{\sqrt{k}}\right)\right)\right|>
D\,\,k^{-7/9},
\end{equation}
for some $D>0$. Integrating by parts in $dt$, we conclude that uniformly in $|\tau|<C_1\,k^{-7/18}$
the contribution of the locus where
and $|\theta|>C_1\,k^{-7/18}$
to the asymptotics of (\ref{eqn:composizione-integrata-riscalata-bis}) is $O\left(k^{-\infty}\right)$.

Now choose $C_2\gg \max\{C_0,C_1\}$. If $|\tau|<C_0\,k^{-7/18}$
and $|\theta|<C_1\,k^{-7/18}$, $|\vartheta| >C_2\,k^{-7/18}$ then
\begin{equation*}
\mathrm{dist}_X\left(\phi^X_{-(\tau_0+\tau/\sqrt{k})}\left(x+\left(\vartheta+\theta,\frac{\mathbf{v}}{\sqrt{k}}\right)\right),
x_{\tau_0}+\frac{\mathbf{w}}{\sqrt{k}}\right)>\frac{C_2}{2}\,k^{-7/18},
\end{equation*}
and so
$$
\big|\partial_u\widetilde{\Psi}_2\big|=
\left|\psi\left(\phi^X_{-(\tau_0+\tau/\sqrt{k})}\left(x+\left(\vartheta+\theta,\frac{\mathbf{v}}{\sqrt{k}}\right)\right),
x_{\tau_0}+\frac{\mathbf{w}}{\sqrt{k}}\right)\right|>
D'\,\,k^{-7/9}.
$$
We now argue as before, using integration by parts in $du$.
\end{proof}

We may thus introduce in (\ref{eqn:composizione-integrata-riscalata-bis}) a cut-off of the form
$\gamma \left(k^{7/18}\,\|(\theta,\vartheta\|)\right)$, where $\gamma \in \mathcal{C}^\infty_0(\mathbb{R})$
is $\ge 0$ and $\equiv 1$ near the origin, perhaps at the cost of losing a rapidly decreasing contribution
(in $\mathcal{C}^j$ norm).

With the rescaling $(\theta,\vartheta)\mapsto (\theta,\vartheta)/\sqrt{k}$, we may then rewrite
(\ref{eqn:composizione-integrata-riscalata-bis}) as follows:
\begin{eqnarray}
\label{eqn:composizione-integrata-riscalata-tris}
\lefteqn{\widetilde{U}_{\tau,k}\left(x+\frac{\mathbf{u}}{\sqrt{k}},x_{\tau_0}+\frac{\mathbf{w}}{\sqrt{k}}\right)}\\
&\sim&\frac{k^{1-\mathrm{d}}}{2\pi}\,\int_{\mathbb{C}^\mathrm{d}}\left[
\int_{1/E}^{E}\int_{1/E}^{E}\int_{-\infty}^{+\infty}\int_{-\infty}^{+\infty}
e^{i\,k\widehat{\Psi}_{2}}\,\widehat{\mathcal{A}}_2\cdot
\mathcal{V}\left(\frac{\theta}{\sqrt{k}},\frac{\mathbf{v}}{\sqrt{k}}\right)\,dt\,du\,d\vartheta\,d\theta\right]\,d\mathbf{v},
\nonumber
\end{eqnarray}
where $\widehat{\Psi}_{2}$ and $\widehat{\mathcal{A}}_2$ are just $\widetilde{\Psi}_{2}$ and $\widetilde{\mathcal{A}}_2$
with the previous rescaling inserted, respectively, and in addition a cut-off $\gamma \left(k^{-1/9}\,\|(\theta,\vartheta)\|\right)$
has been incorporated into $\widehat{\mathcal{A}}_2$. In particular, integration in $d\theta\,d\vartheta$ is over
a ball centered at the origin in $\mathbb{R}^2$, of expanding radius $O\left(k^{1/9}\right)$.

We have (see \S 3 of \cite{sz})
\begin{eqnarray}
\label{eqn:expansion-psi}
\lefteqn{t\,\psi\left(x+\frac{\mathbf{u}}{\sqrt{k}},x+\left(\frac{\theta}{\sqrt{k}},\frac{\mathbf{v}}{\sqrt{k}}\right)\right)}\\
&=&it\,\left[1-e^{-i\theta/\sqrt{k}}\right]-\frac{it}{k}\,\psi_2(\mathbf{u},\mathbf{v})
+t\,R_3\left(\frac{\mathbf{u}}{\sqrt{k}},\frac{\mathbf{v}}{\sqrt{k}}\right)\,e^{-i\theta/\sqrt{k}}\nonumber\\
&=&-\frac{t\,\theta}{\sqrt{k}}+\frac{i\,t}{2k}\,\theta^2-\frac{it}{k}\,\psi_2(\mathbf{u},\mathbf{v})
+t\,R_3\left(\frac{\theta}{\sqrt{k}},\frac{\mathbf{u}}{\sqrt{k}},\frac{\mathbf{v}}{\sqrt{k}}\right).\nonumber
\end{eqnarray}

Using Corollary 2.2 of \cite{pao-torus}, we get from (\ref{eqn:azione-in-coordinate}):
\begin{eqnarray}
\label{eqn:azione-in-coordinate-bis}
\lefteqn{\phi^X_{-(\tau_0+\tau/\sqrt{k})}
\left(x+\left(\frac{1}{\sqrt{k}}\,\left(\vartheta+\theta\right),\frac{\mathbf{v}}{\sqrt{k}}\right)\right)}\\
&=&\phi^X_{-\tau/\sqrt{k}}\left(x_{\tau_0}+\left(\frac{1}{\sqrt{k}}\,\left(\vartheta+\theta\right)
+R_3\left(\frac{\mathbf{v}}{\sqrt{k}}\right),\frac{A\mathbf{v}}{\sqrt{k}}+
R_2\left(\frac{\mathbf{v}}{\sqrt{k}}\right)\right)\right)\nonumber\\
&=&x_{\tau_0}+\big(\Theta_{\tau,k},\Upsilon_{\tau,k}\big),
\end{eqnarray}
where
$$
\Theta_{\tau,k}=:\frac{1}{\sqrt{k}}\,\big(\tau\,f(m)+\vartheta+\theta\big)
+\frac \tau k\,\omega_m\big(\upsilon_f(m),A\mathbf{v}\big)
+R_3\left(\frac {\tau}{\sqrt{k}},\frac{\mathbf{v}}{\sqrt{k}}\right),
$$
$$
\Upsilon_{\tau,k}=:\frac{1}{\sqrt{k}}\,\big(A\mathbf{v}-\tau\,\upsilon_f(m)\big)+R_2\left(\frac \tau {\sqrt{k}},\frac{\mathbf{v}}{\sqrt{k}}\right).
$$
Arguing as for (\ref{eqn:expansion-second-summand}), we now get
\begin{eqnarray}
\label{eqn:expansion-second-summand-bis}
\lefteqn{u\,\psi\left(\phi^X_{-(\tau_0+\tau/\sqrt{k})}
\left(x+\left(\frac{1}{\sqrt{k}}\,\left(\vartheta+\theta\right),\frac{\mathbf{v}}{\sqrt{k}}\right)\right),
x+\frac{\mathbf{w}}{\sqrt{k}}\right)}\\
&=&iu\,\left[1-e^{i\Theta_{\tau,k}}\right]-\frac{i u}{k}\,\psi_2(A\mathbf{v}-\tau\,\upsilon_f(m),\mathbf{w})
\nonumber\\
&&+
R_3\left(\frac{\mathbf{v}}{\sqrt{k}},\frac{\tau}{\sqrt{k}},
\frac{\mathbf{w}}{\sqrt{k}},\frac{\vartheta}{\sqrt{k}}\frac{\theta}{\sqrt{k}}\right)\nonumber\\
&=&\frac{u}{\sqrt{k}}\,\big(\tau\,f(m)+\vartheta+\theta\big)
+\frac uk\,\left[\tau\,\omega_m\big(\upsilon_f(m),A\mathbf{v}\big)+\frac i2\,
\big(\tau\,f(m)+\vartheta+\theta\big)^2\right]\nonumber\\
&&-\frac{i u}{k}\,\psi_2(A\mathbf{v}-\tau\,\upsilon_f(m),\mathbf{w})+
R_3\left(\frac{\mathbf{v}}{\sqrt{k}},
\frac{\mathbf{w}}{\sqrt{k}},\frac{\tau}{\sqrt{k}},\frac{\vartheta}{\sqrt{k}},\frac{\theta}{\sqrt{k}}\right).\nonumber
\end{eqnarray}
Inserting (\ref{eqn:expansion-second-summand-bis}) into (\ref{eqn:composizione-integrata-riscalata-tris}),  we obtain
\begin{eqnarray}
\label{eqn:composizione-integrata-riscalata-quater}
\lefteqn{\widetilde{U}_{\tau,k}\left(x+\frac{\mathbf{u}}{\sqrt{k}},x_{\tau_0}+\frac{\mathbf{w}}{\sqrt{k}}\right)}\\
&\sim&\frac{k^{1-\mathrm{d}}}{2\pi}\,\int_{\mathbb{C}^\mathrm{d}}\left[
\int_{1/E}^{E}\int_{1/E}^{E}\int_{-\infty}^{+\infty}\int_{-\infty}^{+\infty}
e^{i\,\sqrt{k}\,\Psi_\tau}\,\mathcal{A}_\tau\cdot
\mathcal{V}\left(\frac{\theta}{\sqrt{k}},\frac{\mathbf{v}}{\sqrt{k}}\right)\,dt\,du\,d\vartheta\,d\theta\right]\,d\mathbf{v},
\nonumber
\end{eqnarray}
where
$$
\Psi_\tau=:u\big(\tau\,f(m)+\vartheta+\theta\big)-t\,\theta-\vartheta,
$$
while
\begin{eqnarray}
\label{eqn:fase-fluita}
\mathcal{A}_\tau&=:&\exp \left(i\tau\,\omega_m\big(\upsilon_f(m),A\mathbf{v}\big)
-\frac t2\,\theta^2-\frac u2\,
\big(\tau\,f(m)+\vartheta+\theta\big)^2\right)\nonumber\\
&&\cdot \exp\Big(\psi_2(\mathbf{u},\mathbf{v})+\psi_2(A\mathbf{v}-\tau\,\upsilon_f(m),\mathbf{w})\Big)
\cdot \mathcal{A}',
\end{eqnarray}
where $\mathcal{A}'=:\mathcal{A}\cdot e^{ikR_3}$.

We may Taylor expand (\ref{eqn:fase-fluita})
in descending powers of $k^{1/2}$, and
regard the inner integral in (\ref{eqn:composizione-integrata-riscalata-quater}) as an oscillatory integral
in $\sqrt k$, with real phase $\Psi_\tau=\Psi_\tau(t,\theta,u,\vartheta)$ depending on the parameter $\tau$.
Integrating by parts in $dt\,du$, one sees that only a bounded neighborhood of the origin in
the $(\theta,\vartheta)$-plane contributes non-negligibly to the asymptotics (we are assuming
$|\tau|<c$ for some $c>0$); we may then introduce an appropriate cut-off
and assume without loss that integration is compactly
supported.

Now $\Psi_\tau$ has the unique
stationary point
$P_\tau=\big(1,0,1,-\tau\,f(m)\big)$, and
$\Psi_\tau(P_\tau)=\tau\,f(m)$.
Furthermore,the Hessian there is
$$
H(P_\tau)=:\left(
             \begin{array}{cccc}
               0 & -1 & 0 & 0 \\
               -1 & 0 & 1 & 0 \\
               0 & 1 & 0 & 1 \\
               0 & 0 & 1 & 0 \\
             \end{array}
           \right)
$$
for every $\tau$, so that
its signature is zero.
In particular, the stationary phase Lemma yields for the
inner integral in (\ref{eqn:composizione-integrata-riscalata-quater}) an asymptotic expansion in descending powers of
$k^{1/2}$. Integration in $d\mathbf{v}$, on the other hand, is over a ball of radius $O\left(k^{1/9}\right)$,
and the expansion may be integrated term by term.

Summing up, we get for (\ref{eqn:composizione-integrata-riscalata-quater})
an asymptotic expansion of the form
\begin{eqnarray}
\label{eqn:composizione-integrata-riscalata-cinques}
\lefteqn{\widetilde{U}_{\tau,k}\left(x+\frac{\mathbf{u}}{\sqrt{k}},x_{\tau_0}+\frac{\mathbf{w}}{\sqrt{k}}\right)}\\
&\sim&e^{i\sqrt{k}\,\tau\,f(m)}\left[\varrho_{\tau_0}(m)\,\left(\frac k\pi\right)^\mathrm{d}\,\frac{2^{\mathrm{d}}}{\nu(\tau_0,m)}\,
\cdot e^{\mathcal{S}_{\tau_0,m}(\mathbf{u},\mathbf{w})}+O\left(k^{\mathrm{d}-1/2}\right)\right],\nonumber
\end{eqnarray}
where the remainder is a function of $\tau$. The expansion may be differentiated in $\tau$, and the leading order term
of the derivative
at $\tau=0$ is
$$
i\sqrt{k}\,f(m)\,\varrho_\tau(m)\,\left(\frac k\pi\right)^\mathrm{d}\,\frac{2^{\mathrm{d}}}{\nu(\tau,m)}\,
\cdot e^{\mathcal{S}_{\tau,m}(\mathbf{u},\mathbf{w})}.
$$
The statement follows in view of Theorem \ref{thm:concentration-rate} and (\ref{eqn:rescaled-derivative}).


\begin{thebibliography}{Dillo99}



\bibitem[B]{ber}F. A. Berezin, {\em General concept of quantization}, Comm. Math. Phys., \textbf{40} (1975),
153--174



\bibitem[BSZ]{bsz} P. Bleher, B. Shiffman, S. Zelditch, {\em
Universality and scaling of correlations between zeros on complex
manifolds}, Invent. Math. \textbf{142} (2000), 351--395


\bibitem[BG]{bg} L. Boutet de Monvel, V. Guillemin,
{\em The spectral theory of Toeplitz operators},
Annals of Mathematics Studies, \textbf{99} (1981),
Princeton University Press, Princeton, NJ;
University of Tokyo Press, Tokyo


\bibitem[BS]{bs} L. Boutet de Monvel, J. Sj\"ostrand,
{\em Sur la singularit\'e des noyaux de Bergman et de Szeg\"o},
Ast\'erisque \textbf{34-35} (1976), 123--164

\bibitem[C]{c} D. Catlin, {\em
The Bergman kernel and a theorem of Tian},
Analysis and geometry in several complex variables (Katata, 1997), 1-–23, Trends Math.,
Birkh¨auser Boston, Boston, MA, 1999

\bibitem[D]{d} I. Daubechies, {\em Coherent states and projective representation of the linear canonical transformations},
J. Math. Phys. \textbf{21} (1980), no. 6, 1377-–1389


\bibitem[MM1]{mm1} X. Ma, G.Marinescu, {\em Generalized Bergman kernels on symplectic manifolds},
Adv. Math. \textbf{217} (2008), no. 4, 1756-–1815

\bibitem[MM2]{mm2} X. Ma, G.Marinescu, {\em Holomorphic Morse inequalities and Bergman kernels},
Progress in Mathematics \textbf{254},
Birkh¨auser Verlag, Basel, 2007

\bibitem[P1]{pao-trace} R. Paoletti, {\em Szeg\"{o} kernels, Toeplitz operators, and equivariant fixed point formulae},
J. Anal. Math. 106 (2008), 209–236

\bibitem[P2]{pao-torus} R. Paoletti, {\em Asymptotics of Szeg\"{o} kernels under Hamiltonian
torus actions}, arXiv:1006.4273v1, to appear in The Israel Journal of Mathematics

\bibitem[P3]{pao-ltfII} R. Paoletti, {\em Local trace formulae and scaling asymptotics in Toeplitz quantization, II},
arXiv:1103.3303v2

\bibitem[SZ]{sz} B. Shiffman, S. Zelditch, {\em Asymptotics of almost
holomorphic sections of ample line bundles on symplectic
manifolds}, J. Reine Angew. Math. {\bf 544} (2002), 181--222

\bibitem[T]{t} G. Tian,
{\em On a set of polarized K\"{a}hler metrics on algebraic manifolds}, J. Differential Geom. \textbf{32} (1990), no. 1, 99--130

\bibitem[Z1]{z-index} S. Zelditch,
{\em Index and dynamics of quantized contact transformations},
Ann. Inst. Fourier (Grenoble) \textbf{47} (1997), no. 1, 305--363

\bibitem[Z2]{z} S. Zelditch, {\em Szeg\"o kernels and a theorem of Tian},
Int. Math. Res. Not. {\bf 6} (1998), 317--331





\end{thebibliography}
\end{document}